\documentclass[oneside,english]{amsart}
\usepackage{geometry}
\usepackage[T1]{fontenc}
\usepackage[latin9]{inputenc}
\usepackage{amsthm}
\usepackage{amssymb}
\usepackage[all,cmtip]{xy}
\usepackage{hyperref}
\usepackage{tikz}
\usetikzlibrary{shapes.geometric,hobby}
\usepackage{comment}
\usetikzlibrary{decorations.markings}
\usetikzlibrary{decorations.pathmorphing}

\newtheorem{Theorem}{Theorem}[section]
\newtheorem{Lemma}[Theorem]{Lemma}
\newtheorem{Proposition}[Theorem]{Proposition}
\newtheorem{Corollary}[Theorem]{Corollary}
\theoremstyle{definition}
\newtheorem{Definition}[Theorem]{Definition}
\newtheorem{Example}[Theorem]{Example}
\theoremstyle{remark}
\newtheorem{Remark}[Theorem]{Remark}

\makeatletter
\numberwithin{equation}{section}
\numberwithin{figure}{section}

\makeatother

\usepackage{babel}
\usepackage{tikz}

\newcommand{\tlogd}{\TT_{X}(-\log D)}
\newcommand{\OO}{\mathcal{O}}
\renewcommand{\AA}{\mathcal{A}}
\newcommand{\CC}{\mathbb{C}}
\newcommand{\ZZ}{\mathbb{Z}}
\newcommand{\TT}{\mathcal{T}}
\newcommand{\DD}{\mathcal{D}}
\newcommand{\HH}{\mathbb{H}}
\newcommand{\LL}{\mathcal{L}}
\newcommand{\MM}{\mathcal{M}}
\newcommand{\PP}{\mathbb{P}}
\newcommand{\D}{\mathbb{D}}
\newcommand{\wt}{\widetilde}
\newcommand{\dlog}{\mathrm{dlog}}
\newcommand{\ord}{\mathrm{ord}}
\begin{document}

\title{Log Picard algebroids and meromorphic line bundles}

\author{Marco Gualtieri}\email{mgualt@math.toronto.edu} 
\author{Kevin Luk}\email{kevin.kh.luk@gmail.com}

\begin{abstract}
We introduce logarithmic Picard algebroids, a natural class of Lie algebroids adapted to a simple normal crossings divisor on a smooth projective variety.  We show that such algebroids are classified by a subspace of the de Rham cohomology of the divisor complement determined by its mixed Hodge structure.  We then solve the prequantization problem, showing that under the appropriate integrality condition, a log Picard algebroid is the Lie algebroid of symmetries of what is called a meromorphic line bundle, a generalization of the usual notion of line bundle in which the fibres degenerate along the divisor.  We give a geometric description of such bundles and establish a classification theorem for them,  showing that they correspond to a subgroup of the holomorphic line bundles on the divisor complement.  Importantly, these holomorphic line bundles need not be algebraic.  Finally, we provide concrete methods for explicitly constructing examples of meromorphic line bundles, such as for a smooth cubic divisor in the projective plane.   
\end{abstract}

\maketitle
\tableofcontents
\pagebreak
\section{Introduction}	\label{sec 1}

The Atiyah algebroid of a holomorphic line bundle $L$ is the vector bundle whose sections are the infinitesimal symmetries of $L$~\cite{MR0086359}. These symmetries, otherwise known as derivations of $L$ or as $\CC^*$--invariant vector fields on the total space of $L$, form an extension of the tangent sheaf by the sheaf of bundle endomorphisms, which coincides with the structure sheaf in this case.  Such Lie algebroid extensions of the tangent sheaf by the structure sheaf were more generally considered by Be{\u\i}linson and Bernstein~\cite{BeilinsonBernstein93} and called \emph{Picard algebroids}.  A Picard algebroid is the Atiyah algebroid of a line bundle only if it is \emph{prequantizable}, that is, its curvature class satisfies an integrality constraint which relates it to the first Chern class of the line bundle in question. 

We develop a generalization of the above theory in which the tangent sheaf $\TT_X$ of a smooth projective variety $X$ is replaced by the logarithmic tangent sheaf $\TT_X(-\log D)$ of vector fields tangent to a simple normal crossings divisor $D\subset X$.  
In Section~\ref{sec 2}, we introduce and classify logarithmic Picard algebroids, and observe that due to the failure of the Poincar\'e lemma for the log de Rham complex, flat connections may not exist even locally. This leads us to focus on log Picard algebroids which are locally trivial, for which local flat connections do exist.  The classification of these is particularly interesting: we show in Theorem~\ref{loc trivial log PA Hodge thm} that they correspond to a certain intersection of the Hodge and weight filtrations on the cohomology of the complement $U=X\backslash D$.  

Our main problem is then to determine what type of object such logarithmic Picard algebroids may prequantize to, and what the quantization condition is.  In Section~\ref{sec 3} we show that the natural prequantum object is a meromorphic line bundle, a type of line bundle in which the transition functions may have zeros or poles along the divisor $D$.  Meromorphic line bundles have appeared in the literature (see~\cite{MR2915468,MR2600125}), but we did not find a treatment of their basic properties. We make a careful study of these bundles, showing in Section~\ref{local description subsec} that they may be realized as holomorphic $\CC^*$ principal bundles on $U$ whose fibers degenerate over $D$ to an infinite chain of rational curves.  We then classify meromorphic line bundles using the fact that the first Chern class lies in a subgroup of $H^2(U,\ZZ)$ determined by a version of Deligne's mixed Hodge theory with integral coefficients.  Finally, we prove in Theorem~\ref{log PA prequantization theorem} that a log Picard algebroid may be prequantized if and only if its class lies in this subgroup of the integral cohomology of $U$.  Indeed, our results may be viewed in the context of the more general problem of giving a geometric interpretation to integral classes in the various subspaces of cohomology determined by the Hodge and weight filtrations.

In Section~\ref{sec 4}, we use ramified coverings and blowups to provide new methods for explicitly constructing meromorphic line bundles, ending with a concrete construction for the group $\ZZ_3\oplus \ZZ\oplus\ZZ$ of meromorphic line bundles on $\CC P^2$ with pole along a smooth cubic divisor; these constructions are necessarily non-algebraic in nature, as meromorphic line bundles are holomorphic but not necessarily algebraic upon restriction to the divisor complement.

\vspace{.05in}

\noindent \textbf{Acknowledgements.} We would like to thank Ron Donagi, Lisa Jeffrey, Sheldon Katz, Chia-Cheng Liu and Steven Rayan for helpful conversations. M.G. was supported by an NSERC Discovery Grant and K.L. was supported by an Ontario Graduate Scholarship.

\section{Logarithmic Picard algebroids}	\label{sec 2}


Throughout the paper, we work with $(X,D)$ where $X$ is a smooth projective variety over $\mathbb{C}$
 of dimension $n$ and $D$ is a simple normal crossings divisor in $X$.  We denote by $\tlogd$ the logarithmic tangent sheaf, consisting of vector fields tangent to $D$, and by $(\Omega^{\bullet}_{X}(\log D), d)$ the corresponding logarithmic de Rham complex.  Since $D$ has simple normal crossings, $\TT_{X}(-\log D)$ (and hence each $\Omega^{k}_{X}(\log D):=\wedge^{k}\Omega_{X}^{1}(\log D)$) is a locally free $\OO_{X}$-module. The Hodge numbers of $X$ will be denoted by $h^{i,j}:=\dim H^{j}(\Omega_{X}^{i})$. 

We work in the analytic topology in all settings unless otherwise specified.


\begin{Definition}
A log Picard algebroid $(\AA,[\cdot,\cdot],\sigma,e)$ on $(X,D)$ is a locally free $\OO_{X}$-module $\AA$ equipped with a Lie bracket $[\cdot,\cdot]$, a bracket-preserving morphism of $\OO_{X}$-modules $\sigma:\AA\rightarrow \tlogd$, and a central section $e$, such that the Leibniz rule
\begin{equation}\label{leibniz}
[a_{1},fa_{2}]=f[a_{1},a_{2}]+\sigma(a_{1})(f)a_{2}
\end{equation}
holds for all $f\in\OO_{X}$, $a_{1},a_{2}\in \AA$, and the sequence 
\begin{equation}\label{exlalg}
0\rightarrow \OO_{X} \overset{e}{\rightarrow} \AA \overset{\sigma}{\rightarrow} \tlogd \rightarrow 0 
\end{equation}
is exact. We usually denote the quadruple $(\AA,[\cdot,\cdot],\sigma,e)$ simply by $\AA$. 
\end{Definition}

\begin{Remark}
When the divisor is empty, we recover the notion of a Picard algebroid as given in \cite{BeilinsonBernstein93}. 
\end{Remark}

Log Picard algebroids over a fixed base $(X,D)$ form a Picard category, with monoidal structure given as follows. For $(\AA_i,[\cdot,\cdot]_i, \sigma_i, e_i)_{i=1,2}$ log Picard algebroids, we define the Baer sum
\begin{equation}\label{monoidal1}
\AA_1 \boxplus \AA_2 = (\AA_1\oplus_{\TT_{X}(-\log D)} \AA_2) / \left<(e_1,-e_2)\right>,
\end{equation}
equipped with the morphism $\sigma_1+\sigma_2$ to $\TT_{X}(-\log D)$, componentwise Lie bracket, and the central section $[(e_1,e_2)]$.  
\begin{Definition}
A \textit{\textbf{connection}} on a log Picard algebroid $\AA$ is a splitting 
\[
\nabla: \tlogd\rightarrow \AA
\]
of the sequence~\eqref{exlalg}. The sheaf of connections therefore forms a sheaf of affine spaces over $\Omega^1_X(\log D)$. 
Given connections $\nabla,\nabla'$ on $\AA,\AA'$, the sum $\nabla +\nabla'$ defines a natural connection on the Baer sum $\AA \boxplus \AA'$, so that the abovementioned Picard category structure extends to log Picard algebroids with connection. 

The \textit{\textbf{curvature}} associates to any connection $\nabla$ the 2-form $F_{\nabla}\in\Omega^{2}_{X}(\log D)$ defined by
\begin{equation}\label{curvdef}
[\nabla\xi_{1},\nabla\xi_{2}]-\nabla[\xi_{1},\xi_{2}] = F_{\nabla}(\xi_{1},\xi_{2}) e,
\end{equation}
for $\xi_{1},\xi_{2}\in \tlogd$.
\end{Definition}

\begin{Lemma}\label{curvtr}
The curvature $F_{\nabla}$ is a closed logarithmic 2-form.   Translating the connection by $A\in\Omega^1_{X}(\log D)$ gives the curvature  
\begin{equation}\label{torscurv}
F_{\nabla + A} = F_{\nabla} + dA,
\end{equation}
and if $\psi:\AA\to\AA'$ is an isomorphism of log Picard algebroids, then 
$F_{\psi\circ\nabla} = F_\nabla$.  Finally, if $\nabla,\nabla'$ are connections for $\AA,\AA'$, then the curvature of the Baer sum is $F_{\nabla + \nabla'} = F_{\nabla} + F_{\nabla'}$.
\end{Lemma}

\begin{proof}
The Leibniz rule~\ref{leibniz} together with the centrality of $e$ imply that $\xi(f) = [\nabla \xi, fe]$ for $\xi\in \TT_{X}(-\log D)$, $f\in\OO_{X}$.  Using this to compute the exterior derivative, we obtain
\begin{equation*}
\begin{aligned}
dF_{\nabla}(\xi_{1},\xi_{2},\xi_{3}) &=[\nabla \xi_{1},[\nabla \xi_{2},\nabla \xi_{3}] - \nabla [\xi_{2},\xi_{3}]] +\cdots \\
&\qquad -([\nabla[\xi_{1},\xi_{2}], \nabla \xi_{3}] - \nabla[[\xi_{1},\xi_{2}],\xi_{3}]) + \cdots,
\end{aligned}
\end{equation*}
where the omitted terms are cyclic permutations of the preceding terms. When taken with their cyclic permutations, the second and third displayed summands cancel; similarly, the first and fourth summands vanish using the Jacobi identity.  For property~\eqref{torscurv}, if $\nabla' = \nabla + A$, then by~\eqref{curvdef} we have 
\begin{equation*}
F_{\nabla'}(\xi_1,\xi_2)  = F_{\nabla}(\xi_1,\xi_2) + [\nabla\xi_1,A\xi_2] + [A\xi_1,\nabla\xi_2] - A[\xi_1,\xi_2] = (F_\nabla + dA)(\xi_1,\xi_2),
\end{equation*}
as required.  For the remaining statements, note firstly that an isomorphism of log Picard algebroids is a bracket-preserving isomorphism $\psi:\AA\to\AA'$ of extensions, giving $F_{\psi\circ\nabla}=F_\nabla$ immediately.  The final assertion follows immediately from the fact that the Lie bracket on the Baer sum $\AA\boxplus\AA'$ is defined componentwise. 
\end{proof}

\begin{Proposition}\label{splitting proposition}
The splitting determined by the connection $\nabla$ takes the Lie bracket on $\AA$ to the bracket on $\TT_{X}(-\log D)\oplus \OO_{X}$ given by
\begin{equation}	\label{eq: split ses} 
[\xi_{1}+f_{1},\xi_{2}+f_{2}]=[\xi_{1},\xi_{2}]+\xi_{1}(f_{2})-\xi_{2}(f_{1})+F_{\nabla}(\xi_{1},\xi_{2}),
\end{equation}
where $\xi_{1},\xi_{2}\in \tlogd$ and $f_{1},f_{2}\in \OO_{X}$. 
\end{Proposition} 

\begin{proof}
The claim follows immediately from the identity
\begin{eqnarray*}
[\nabla\xi_{1}+f_{1}e,\nabla\xi_{2}+f_{2}e] & = & 
[\nabla(\xi_{1}),\nabla(\xi_{2})]+[\nabla(\xi_{1}),f_{2}e]-[\nabla(\xi_{2}),f_{1}e]+[f_{1}e,f_{2}e]\\
 & = & \nabla[\xi_{1},\xi_{2}] + F_{\nabla}(\xi_{1},\xi_{2})e + \xi_{1}(f_{2})e - \xi_{2}(f_{1})e.
\end{eqnarray*}
\end{proof}

\begin{Definition} \label{twist definition}
Let $B\in\Omega_{X}^{2,\mathrm{cl}}(\log D)$ be a closed logarithmic 2-form and $\AA$ be a log Picard algebroid on $(X,D)$. We define the \textbf{\textit{twist}} of $\AA$ to be the log Picard algebroid $(\AA,[\cdot,\cdot]_{B},\sigma,e)$ where $[\cdot,\cdot]_{B}$ is given as
\[
[a_1,a_2]_B = [a_1,a_2] + B(\sigma(a_1),\sigma(a_2))e,
\]
for $a_{1},a_{2}\in\AA$. We denote this algebroid by $\AA_B$.
\end{Definition}


\begin{Example}[Split log Picard algebroid]
\label{global split log PA}
Let $B\in\Omega_{X}^{2,\mathrm{cl}}(\log D)$ be a closed logarithmic 2-form. We can twist the trivial extension $\tlogd\oplus\OO_{X}$ by $B$
\[
[\xi_{1}+f_{1},\xi_{2}+f_{2}]_B=[\xi_{1},\xi_{2}]+\xi_{1}(f_{2})-\xi_{2}(f_{1})+B(\xi_{1},\xi_{2}).
\]
Hence, we obtain a log Picard algebroid admitting a global connection with curvature $B$. 
\end{Example}

\subsection{Mixed Hodge theory}\label{mht}
In this section, we recall how the logarithmic de Rham complex of $(X,D)$ is used to describe the de Rham cohomology of the complement $U = X\backslash D$ and its weight filtration. By \cite[Proposition 3.1.8]{Deligne71}, the restriction of logarithmic forms to 
usual forms on $U$  defines a quasi-isomorphism of complexes
\begin{equation}	\label{eq:quasi-iso from log to complement}
\Omega_{X}^{\bullet}(\log D)\hookrightarrow Rj_{*}\Omega_{U}^{\bullet},
\end{equation}
where $j:U\hookrightarrow X$ is the inclusion. The log de Rham complex therefore computes the cohomology of $U$:
\[
H^{k}(U,\mathbb{C})=\mathbb{H}^{k}(X,\Omega_{X}^{\bullet}(\log D)).
\]
The $\textit{\textbf{Hodge filtration}}$ on $\Omega_{X}^{\bullet}(\log D)$ is defined as follows:
\[
F^{p}\Omega^{\bullet}_{X}(\log D):=\Omega_{X}^{\geq p}(\log D),
\]
where $\Omega_{X}^{\geq p}(\log D)$ is the truncated complex
\[
\Omega_{X}^{\geq p}(\log D):=\dots\rightarrow0\rightarrow\Omega_{X}^{p}(\log D)\rightarrow\dots\rightarrow\Omega_{X}^{n}(\log D).
\]
This provides a Hodge filtration on $H^{k}(U,\mathbb{C})$, as follows:
\[
F^{p}H^{k}(U,\mathbb{C}):=\mathrm{im}(\mathbb{H}^{k}(X,F^{p}\Omega^{\bullet}_{X}(\log D))\rightarrow H^{k}(U,\mathbb{C})).
\]
We now briefly review Deligne's theory of weight filtrations on $H^{k}(U,\mathbb{C})$, following the treatment in \cite{Deligne71}. The $\textit{\textbf{weight filtration}}$ on $\Omega_{X}^{\bullet}(\log D)$ is defined as follows:
\[
W_{m}\Omega_{X}^{k}(\log D):=\begin{cases}
0 & m<0\\
\Omega_{X}^{k}(\log D) & m\geq k\\
\Omega_{X}^{k-m}\wedge\Omega_{X}^{m}(\log D) & 0\leq m\leq k,
\end{cases}
\]
inducing a corresponding weight filtration on $H^{k}(U,\mathbb{C})$, namely
\begin{equation}	\label{eq:weights}
W_{m}H^{k}(U,\mathbb{C}):=\mathrm{im}(\mathbb{H}^{k}(X,W_{m-k}\Omega^{\bullet}_{X}(\log D))\rightarrow H^{k}(U,\mathbb{C})).
\end{equation} 
Moreover, we can describe the associated graded pieces of the weight filtration, $\mathrm{Gr}_{m}^{W}(\Omega_{X}^{\bullet}(\log D))$, using the Poincar\'e residue isomorphism \cite[3.1.5.2]{Deligne71}.
\begin{equation}	\label{Poincare residue}
\mathrm{Gr}_{m}^{W}(\Omega_{X}^{\bullet}(\log D))\overset{\cong}{\rightarrow}j_{m_{*}}\Omega_{D^{(m)}}^{\bullet}[-m],
\end{equation}
where $D^{(m)}$ denotes the disjoint union of all $m$-fold intersections of different irreducible components of the divisor $D$ with $D^{0}=X$ and $j_{m}:D^{(m)}\hookrightarrow X$ is the natural inclusion. There is a spectral sequence associated to the weight filtration $W_m$ on $\Omega_{X}^{\bullet}(\log D)$
\begin{equation*}
\begin{split}
E_{1}^{-m,k+m} & =\HH^{(k+m)+(-m)}(X,\mathrm{Gr}^W_{-(-m)}(\Omega_{X}^{\bullet}(\log D))) \\
& = \mathbb{H}^{k}(X,\mathrm{Gr}_{m}^{W}(\Omega_{X}^{\bullet}(\log D)))\Rightarrow H^{k}(U,\mathbb{C})
\end{split}
\end{equation*}
which we call the weight spectral sequence. Using \eqref{Poincare residue}, we can simplify and rewrite the weight spectral sequence as
\begin{equation}	\label{eq: weight spectral sequence}
E_{1}^{-m,k+m}=H^{k-m}(D^{(m)},\mathbb{C})\Rightarrow H^{k}(U,\mathbb{C}).
\end{equation}
Moreover, the differential $d_1$ of the weight spectral sequence is given by
\begin{equation} \label{eq:d1 weight spectral sequence}
d_1:E_{1}^{-m,k+m} = H^{k-m}(D^{(m)},\mathbb{C})\rightarrow H^{k-m+2}(D^{(m-1)},\CC) = E_{1}^{-m+1,k+m}.
\end{equation}
By \cite[Corollary 3.2.13]{Deligne71}, the weight spectral sequence degenerates at $E_{2}$. The associated graded objects of the weight filtration are exactly the $E_{2}$-terms
\[
E_{2}^{-m,k+m}=\mathrm{Gr}_{k+m}^{W}(H^{k}(U,\mathbb{C})).
\]
We now provide an alternative method to define the weight filtration on $H^{k}(U,\mathbb{C})$. This will be useful when we generalize to weight filtrations on the cohomology with integer coefficients in Section \ref{sec 3}. First, we define the $\textit{\textbf{canonical filtration}}$ $\sigma_{\leq p}$ of a complex $K^{\bullet}$ as
\[
\sigma_{\leq p}K^{\bullet}:=\dots\rightarrow K^{p-1}\rightarrow\ker(d)\rightarrow0\rightarrow0\rightarrow\dots
\]
where the graded pieces $\mathrm{Gr}_{p}^{\sigma}K^{\bullet}$ are given as
\[
\mathrm{Gr}_{p}^{\sigma}K^{\bullet}=H^p(K^\bullet)[-p].
\]
By \cite[Lemma 4.9]{PetersSteenbrink08}, there is a filtered quasi-isomorphism of complexes
\[
(\Omega_{X}^{\bullet}(\log D),\sigma)\rightarrow(\Omega_{X}^{\bullet}(\log D),W).
\]
Furthermore, since the complexes $Rj_{*}\mathbb{C}_{U}$ and $\Omega_{X}^{\bullet}(\log D)$ are quasi-isomorphic from \eqref{eq:quasi-iso from log to complement}, we can alternatively write the weight filtration on $H^{k}(U,\mathbb{C})$ in \eqref{eq:weights} as
\[
W_{m}H^{k}(U,\mathbb{C}):=\mathrm{im}(\mathbb{H}^{k}(\sigma_{\leq m-k}Rj_{*}\mathbb{C}_{U})\rightarrow H^{k}(U,\mathbb{C})).
\] 
We now explain how to recover the weight spectral sequence in \eqref{eq: weight spectral sequence}. The Leray spectral sequence associated to the canonical filtration $\sigma$ on the complex $Rj_{*}\mathbb{C}_{U}$ is given by
\begin{equation}	\label{eq:E2 Leray}
E_{2}^{k-m,m}=H^{k-m}(R^{m}j_{*}\mathbb{C}_{U})\Rightarrow H^{k}(U,\mathbb{C})
\end{equation}
Since there is a filtered quasi-isomorphism between the complexes $(Rj_{*}\mathbb{C}_{U},\sigma)$ and $(\Omega_{X}^{\bullet}(\log D),W)$, we obtain the following isomorphism between the graded pieces
\[
\mathrm{Gr}^{\sigma}_{m}Rj_{*}\mathbb{C}_{U}\cong\mathrm{Gr}^{W}_{m}\Omega_{X}^{\bullet}(\log D)
\]
which itself yields the following
\begin{equation} \label{eq:R^k interpretation for C}
R^{m}j_{*}\mathbb{C}_{U}\cong j_{m_{*}}\Omega_{D^{(m)}}^{\bullet}
\end{equation}
Hence, we can use this to re-write the $E_{2}$-terms of the Leray spectral sequence in \eqref{eq:E2 Leray} as
\[
E_{2}^{k-m,m}=H^{k-m}(D^{(m)},\mathbb{C})\Rightarrow H^{k}(U,\mathbb{C})
\]
which is exactly the weight spectral sequence as described in \eqref{eq: weight spectral sequence}.
\begin{Remark}	\label{smooth two step weight remark}
For the special case where $D$ is a smooth irreducible divisor on $X$, the weight filtration on the cohomology of the complement $U$ is a two-step filtration
\[
H^{k}(U,\mathbb{C})=W_{k+1}\supset W_{k}\supset0
\]
which can be described by the complex Gysin sequence
\[
\dots\rightarrow H^{k}(X,\mathbb{C})\rightarrow H^{k}(U,\mathbb{C})\rightarrow H^{k-1}(D,\mathbb{C})\rightarrow H^{k+1}(X,\mathbb{C})\rightarrow\dots
\]
The $W_{k}$-piece is given by
\[
W_{k}H^{k}(U,\mathbb{C})=\mathrm{im}(H^{k}(X,\mathbb{C})\rightarrow H^{k}(U,\mathbb{C})),
\]
and the associated graded piece $\mathrm{Gr}_{k+1}^{W}H^{k}(U,\mathbb{C})$ is given by
\[
\mathrm{Gr}_{k+1}^{W}H^{k}(U,\mathbb{C})=\ker(H^{k-1}(D,\mathbb{C})\rightarrow H^{k+1}(X,\mathbb{C})).
\]
\end{Remark}

\subsection{Classification of log Picard algebroids}
We now classify log Picard algebroids using the Hodge-theoretic techniques introduced above.

\begin{Theorem} \label{classification theorem}
Log Picard algebroids on $(X,D)$ are classified by the hypercohomology group:
\[
\mathbb{H}^{2}(F^{1}\Omega^{\bullet}_{X}(\log D)), 
\]
with addition deriving from the Baer sum operation \eqref{monoidal1}.
\end{Theorem}

\begin{proof}
Let $\mathcal{A}$ be a log Picard algebroid on $(X,D)$ and take an open affine cover $\{U_{i}\}$ of $X$. Then over each $U_i$ we may choose a connection $\nabla_i$, which has curvature $B_i\in \Omega^{2,\mathrm{cl}}(U_i,\log D)$.  On the double overlap $U_i\cap U_j$, we have by Lemma~\ref{curvtr} that 
\[
B_{i}-B_{j}=dA_{ij},
\]
where $A_{ij} = \nabla_i - \nabla_j\in \Omega^1(U_i\cap U_j,\log D)$.  Since $A_{ij}+A_{jk}+A_{ki}=0$ on triple overlaps, we have a \v{C}ech-de Rham cocycle $(A_{ij}, B_i)$.  Modifying our initial choice of local connections shifts this cocycle by a coboundary, so that we have a well-defined map from log Picard algebroids to the hypercohomology group $\mathbb{H}^{2}(F^{1}\Omega_{X}^{\bullet}(\log D))$. An isomorphism $\psi:\AA\to \AA'$ of log Picard algebroids induces an isomorphism of sheaves of connections which preserves the curvature (Lemma~\ref{curvtr}), so that $\AA$ and $\AA'$ give rise to cohomologous cocycles by the above argument.  

To show that this classifying map is surjective, note that any such cocycle $(A_{ij},B_i)$ may be used to construct a log Picard algebroid as follows: on each $U_i$ we define local log Picard algebroids $\AA_i$ to be trivial extensions $\AA_i = \mathcal{O}_{U_{i}}\oplus\TT_{U_{i}}(-\log D)$ equipped with brackets as in Proposition~\ref{splitting proposition}:
\[
[f_{1}+\xi_{1},f_{2}+\xi_{2}]_i=[\xi_{1},\xi_{2}]+\xi_{1}(f_{2})-\xi_{2}(f_{1})+B_{i}(\xi_{1},\xi_{2}).
\]
We then glue $\AA_i$ to $\AA_j$ over $U_i\cap U_j$ using the isomorphism of log Picard algebroids 
\[
\Phi_{ij}:f+ \xi \mapsto f+ \xi + A_{ij}(\xi),
\]
which satisfies $\Phi_{ki}\Phi_{jk}\Phi_{ij}=1$, yielding a global log Picard algebroid and so proving surjectivity. To obtain the correspondence between Baer sum and addition in cohomology, note that by choosing local connections $\nabla_i, \nabla'_i$  on $\AA, \AA'$, the sum $\nabla_i+\nabla_i'$ defines local connections for the Baer sum $\AA\boxplus\AA'$ and by Lemma~\ref{curvtr} the corresponding cocycle for $\nabla_i+\nabla_i'$ is the sum of the cocycles for the summands, as required.
\end{proof}

\begin{Example} [Family of split log Picard algebroids]
Given two closed logarithmic 2-forms $B,B'\in\Omega_{X}^{2,\mathrm{cl}}(\log D)$, we can twist the trivial extension $\tlogd\oplus\OO_X$ by $B$ and $B'$ as shown in Example~\ref{global split log PA} to obtain split log Picard algebroids $\AA_B$ and $\AA_{B'}$ respectively. An automorphism of the trivial extension given by $\Phi:\xi+f\mapsto \xi + f + A(\xi)$ for $A\in H^0(\Omega^1_X(\log D))$ takes $[\cdot,\cdot]_B$ to $[\cdot,\cdot]_{B'}$ if and only if $B'-B=dA$. But by~\cite[Proposition 1.3.2]{Deligne71}, any global closed logarithmic form on $(X,D)$ is exact if and only if it vanishes, hence we see that $\AA_B\cong \AA_{B'}$ if and only if $B'=B$. Therefore we obtain a universal family of split log Picard algebroids over $H^0(\Omega^{2,\mathrm{cl}}_X(\log D))$ by this construction.
\end{Example}

\begin{Theorem}	\label{log PA classification theorem}
Any extension of $\mathcal{T}_{X}(-\log D)$ by $\mathcal{O}_{X}$ admits the structure of a log Picard algebroid, and if $\AA,\AA'$ are two log Picard algebroids with isomorphic underlying extensions, then $\AA'$ is isomorphic to $\AA$ up to a twist by a closed logarithmic 2-form as in Definition~\ref{twist definition}.
\end{Theorem}

\begin{proof}
Given a log Picard algebroid described by the cocycle $\{A_{ij},B_{i}\}$, the cocycle $\{A_{ij}\}$ represents the underlying extension class of the log Picard algebroid. The forgetful map from $\{A_{ij},B_{i}\}$ to $\{A_{ij}\}$ induces the map in cohomology
\[
\Phi:\mathbb{H}^{2}(F^{1}\Omega^{\bullet}_{X}(\log D))\rightarrow H^{1}(\Omega_{X}^{1}(\log D)).
\]
To prove our theorem, we will show that this map fits into the following short exact sequence
\[
0\rightarrow H^{0}(\Omega_{X}^{2,\mathrm{cl}}(\log D))\rightarrow\mathbb{H}^{2}(F^{1}\Omega^{\bullet}_{X}(\log D))\overset{\Phi}{\rightarrow}H^{1}(\Omega_{X}^{1}(\log D))\rightarrow0.
\]
The spectral sequence associated to the Hodge filtration on $\Omega^{\bullet}_{X}(\log D)$ is exactly the logarithmic Hodge-de Rham spectral sequence, and by \cite[Corollary 3.2.13]{Deligne71}, we know that this spectral sequence degenerates at $E_{1}$. In addition, since the log de Rham complex $\Omega_{X}^{\bullet}(\log D)$ with the Hodge filtration is a biregular filtered complex and so by \cite[Proposition 1.3.2]{Deligne71}, the following sequence in hypercohomology is short exact
\begin{equation}	\label{eq: F2 short exact sequence}
0\rightarrow\mathbb{H}^{2}(F^{2}\Omega_{X}^{\bullet}(\log D))\rightarrow\mathbb{H}^{2}(F^{1}\Omega_{X}^{\bullet}(\log D))\rightarrow H^{1}(\Omega_{X}^{1}(\log D))\rightarrow0
\end{equation}
The group $\mathbb{H}^{2}(F^{2}\Omega_{X}^{\bullet}(\log D))$ can be identified with $H^{0}(\Omega_{X}^{2,\mathrm{cl}}(\log D))$ which completes the proof.
\end{proof}

\begin{Definition}	\label{loc trivial PA def}
Let $\mathcal{A}$ be a log Picard algebroid on $(X,D)$. We say that $\mathcal{A}$ is \textbf{\textit{locally trivial}} if a flat connection (a splitting of $\mathcal{A}$ with zero curvature) exists in a neighbourhood of every point. 
\end{Definition}

\begin{Proposition}	\label{loc trivial PA cohomology}
Locally trivial log Picard algebroids on $(X,D)$ are classified by the subgroup
\[
H^{1}(\Omega_{X}^{1,\mathrm{cl}}(\log D))\subset\mathbb{H}^{2}(F^{1}\Omega^{\bullet}_{X}(\log D)).
\]
If $D$ is smooth, the above inclusion is an equality, so that all log Picard algebroids on $(X,D)$ are locally trivial. 
\end{Proposition}

\begin{proof}
As described in the proof of Theorem \ref{classification theorem}, a log Picard algebroid is classified by the cohomological data $\{ A_{ij},B_{i}\}$ where $B_{i}-B_{j}=dA_{ij}$. In the locally trivial case, local flat splittings may be chosen, in which case $B_{i}$ and $B_{j}$ are zero, yielding a \v{C}ech 1-cocycle $A_{ij}$ for $\Omega_{X}^{1,\mathrm{cl}}(\log D)$, as required.  The inclusion of complexes 
\[
\Omega_{X}^{1,\mathrm{cl}}(\log D)\hookrightarrow[\Omega_{X}^{1}(\log D)\rightarrow\Omega_{X}^{2,\mathrm{cl}}(\log D)]
\]
induces an inclusion on first cohomology groups, and since the Poincar\'e lemma for logarithmic forms holds in degree $\geq 2$ when $D$ is smooth, we obtain surjectivity in that case.
\end{proof}

%


A log Picard algebroid on $(X,D)$ restricts to a usual Picard algebroid on $U=X\setminus D$, and so determines a class in $H^1(\Omega^{1,\text{cl}}_U)$, which itself defines a class in $H^2(U,\CC)$.  We now give a classification of locally trivial log Picard algebroids in terms of this cohomology class.

%

\begin{Theorem}	\label{loc trivial log PA Hodge thm} 
The restriction of a locally trivial log Picard algebroid on $(X,D)$ to the complement $U=X\setminus D$ defines a cohomology class in the subgroup  
\[
F^{1}H^{2}(U,\mathbb{C})\cap W_{3}H^{2}(U,\mathbb{C}) \subset H^2(U,\CC),
\]
and this defines a bijection on isomorphism classes.
\end{Theorem}

\begin{proof}
By Proposition \ref{loc trivial PA cohomology}, locally trivial log Picard algebroids on $(X,D)$ are classified by the cohomology group $H^{1}(\Omega_{X}^{1,\mathrm{cl}}(\log D))$. Our goal is to prove the following
\begin{equation}	\label{eq: loc trivial log PA hodge}
H^{1}(\Omega_{X}^{1,\mathrm{cl}}(\log D))\cong F^{1}H^{2}(U,\mathbb{C})\cap W_{3}H^{2}(U,\mathbb{C}),
\end{equation}
using the description of the cohomology of $U$ in terms of logarithmic forms given in Section~\ref{mht}.
First, we observe that the sheaf $\Omega_{X}^{1,\mathrm{cl}}(\log D)$ can be resolved by the following complex
\[
\Omega_{X}^{1}(\log D)\rightarrow\Omega_{X}^{1}\wedge\Omega_{X}^{1}(\log D)\rightarrow\Omega_{X}^{2}\wedge\Omega_{X}^{1}(\log D)\rightarrow\dots
\]
Using the definition of the weight filtration on $\Omega_{X}^{\bullet}(\log D)$, we see that this complex is $F^{1}W_{1}\Omega_{X}^{\bullet}(\log D)$, the first truncation of $W_{1}\Omega_{X}^{\bullet}(\log D)$. Hence, we have the isomorphism
\[
H^{1}(\Omega_{X}^{1,\mathrm{cl}}(\log D)) \cong \HH^{2}(F^{1}W_{1}\Omega_{X}^{\bullet}(\log D)).
\]
Now, consider the commutative diagram of complexes
\[
\xymatrix{
F^{1}W_{1}\Omega_{X}^{\bullet}(\log D)\ar[r]\ar[d] & W_{1}\Omega_{X}^{\bullet}(\log D)\ar[d] \\
F^{1}\Omega_{X}^{\bullet}(\log D)\ar[r] & \Omega_{X}^{\bullet}(\log D)}
\]
We will now show that if we take the second hypercohomology $\mathbb{H}^{2}$ of all the complexes in the above diagram, then all of the induced maps are injective. This will prove that the map
\begin{equation} \label{eq:HH2 into F1W3}
\HH^{2}(F^{1}W_{1}\Omega_{X}^{\bullet}(\log D))\rightarrow F^{1}H^{2}(U,\CC)\cap W_{3}H^{2}(U,\CC),
\end{equation}
is injective. By the $E_{1}$-degeneration of the logarithmic Hodge-de Rham spectral sequence and the fact that $\Omega_{X}^{\bullet}(\log D)$ with the Hodge filtration is a biregular filtered complex, we obtain that the map
\[
\mathbb{H}^{2}(F^{1}\Omega_{X}^{\bullet}(\log D))\rightarrow\mathbb{H}^{2}(\Omega_{X}^{\bullet}(\log D))
\]
is injective \cite[Proposition 1.3.2]{Deligne71}. We now prove the injectivity of the map
\begin{equation}	\label{eq:H2(W1) to H2 injection}
\mathbb{H}^{2}(W_{1}\Omega_{X}^{\bullet}(\log D))\rightarrow\mathbb{H}^{2}(\Omega_{X}^{\bullet}(\log D)).
\end{equation}
To do this, we first consider the short exact sequence of complexes
\[
0\rightarrow W_{1}\Omega_{X}^{\bullet}(\log D)\rightarrow W_{2}\Omega_{X}^{\bullet}(\log D)\rightarrow\mathrm{Gr}_{2}^{W}\Omega_{X}^{\bullet}(\log D)\rightarrow0
\]
with the associated hypercohomology long exact sequence
\begin{equation} \label{weight hypercohomology sequence}
\dots\rightarrow\mathbb{H}^{1}(\mathrm{Gr}_{2}^{W}\Omega_{X}^{\bullet}(\log D))\rightarrow\mathbb{H}^{2}(W_{1}\Omega_{X}^{\bullet}(\log D))\rightarrow\mathbb{H}^{2}(W_{2}\Omega_{X}^{\bullet}(\log D))\rightarrow\dots,
\end{equation}
Recall from the Poincar\'e residue isomorphism in \eqref{Poincare residue}, we have the following description of the associated graded objects of the weight filtration 
\[
\mathrm{Gr}_{2}^{W}\Omega_{X}^{\bullet}(\log D)\cong j_{2_{*}}\Omega_{D^{(2)}}^{\cdot}[-2],
\]
and hence $\mathbb{H}^{1}(\mathrm{Gr}_{2}^{W}\Omega_{X}^{\bullet}(\log D))=0$. Therefore, we have the injectivity of the map
\[
\mathbb{H}^{2}(W_{1}\Omega_{X}^{\bullet}(\log D))\rightarrow\mathbb{H}^{2}(W_{2}\Omega_{X}^{\bullet}(\log D)).
\]
But the inclusion of $W_k\Omega_X^\bullet(\log D)$ in $\Omega_X^\bullet(\log D)$ is a quasi-isomorphism up to and including degree $k$, and so the map \eqref{eq:H2(W1) to H2 injection} must be injective as well. Using the same argument as above, we obtain also that the following map
\[
\mathbb{H}^{2}(F^{1}W_{1}\Omega_{X}^{\bullet}(\log D))\rightarrow\mathbb{H}^{2}(F^{1}\Omega_{X}^{\bullet}(\log D))
\]
is injective. To conclude the proof of the theorem, we need to show that the map earlier in \eqref{eq:HH2 into F1W3} is surjective. The group $\HH^{2}(F^{1}W_{1}\Omega_{X}^{\bullet}(\log D))$ injects into $\HH^{2}(W_{1}\Omega_{X}^{\bullet}(\log D))$ with cokernel $H^{2}(\OO_{X})$,
\begin{equation} \label{eq:F1W1 into W1}
0\rightarrow\HH^{2}(F^{1}W_{1}\Omega_{X}^{\bullet}(\log D))\rightarrow\HH^{2}(W_{1}\Omega_{X}^{\bullet}(\log D))\rightarrow H^{2}(\OO_{X})\rightarrow0,
\end{equation}
and it injects into $\HH^{2}(F^{1}\Omega_{X}^{\bullet}(\log D))$ with cokernel $H^{0}(D^{(2)},\CC)$,
\begin{equation} \label{eq:F1W1 into F1}
0\rightarrow\HH^{2}(F^{1}W_{1}\Omega_{X}^{\bullet}(\log D))\rightarrow\HH^{2}(F^{1}\Omega_{X}^{\bullet}(\log D))\rightarrow H^{0}(D^{(2)},\CC)\rightarrow0.
\end{equation}
Suppose that there exists a class in $F^{1}H^{2}(U,\CC)\cap W_{3}H^{2}(U,\CC)$ such that it is not in the image of $\HH^{2}(F^{1}W_{1}\Omega_{X}^{\bullet}(\log D))$ under the map \eqref{eq:HH2 into F1W3}. Then, this class must either map to $H^{2}(\OO_{X})$ or map to $H^{0}(D^{(2)},\CC)$. However, if this class maps to $H^{2}(\OO_{X})$, then this contradicts that the class lies in $F^{1}H^{2}(U,\CC)$. If this class maps to $H^{0}(D^{(2)},\CC)$, then this contradicts that the class lies in $W_{3}H^{2}(U,\CC)$. Thus, we obtain the surjectivity of the map in \eqref{eq:HH2 into F1W3}.
\end{proof}

\begin{Example}	\label{Deligne line bundle example}
Consider $(\mathbb{P}^{2},D)$ where $D$ is the union of three projective lines in triangle position. The divisor complement is $U=\mathbb{C}^{*}\times\mathbb{C}^{*}=\mathrm{Spec}\,\CC[x^{\pm 1},y^{\pm 1}]$ and we have $H^{2}(U,\mathbb{C})=\mathbb{C}$. If we write out the $k+m=4$ row of the $E_{1}^{-m,k+m}$-page of the weight spectral sequence given in~\eqref{eq: weight spectral sequence} with the differentials $d_1$ given in~\eqref{eq:d1 weight spectral sequence}, we have 
\[
0\overset{d_{1}}{\rightarrow}\underbrace{H^{0}(D^{(2)},\mathbb{C})}_{\mathbb{C}^{3}}\overset{d_{1}}{\rightarrow}\underbrace{H^{2}(D,\mathbb{C})}_{\mathbb{C}^{3}}\overset{d_{1}}{\rightarrow}\underbrace{H^{4}(X,\mathbb{C})}_{\mathbb{C}}\overset{d_{1}}{\rightarrow}0
\]
We obtain that the $k+m=4$ row of the $E_{2}^{-m,k+m}$-page of the weight spectral sequence is
\[
0\rightarrow\mathbb{C}\rightarrow0\rightarrow0\rightarrow0
\]
The graded piece $\mathrm{Gr}^{W}_{4}H^{2}(U,\mathbb{C})$ is the $E_{2}^{-2,4}$-term of the weight spectral sequence and so we have:
\[
\mathrm{Gr}^{W}_{4}H^{2}(U,\mathbb{C})=\mathbb{C}
\]
This implies that $W_{3}H^{2}(U,\mathbb{C})$ is necessarily zero and so by the previous theorem, the log Picard algebroids for this example cannot be locally trivial. We now present a non-trivial log Picard algebroid for this example.  At a vertex of the triangle, $D$ can be given locally as $\{xy=0\}$. If we consider the following closed log 2-form $B$,
\[
B:=\frac{1}{(2\pi i)^{2}}\frac{dx}{x}\wedge\frac{dy}{y}
\]
this is not locally exact near the vertex of the triangle. If we twist the Lie bracket on the trivial log Picard algebroid $\TT_{\PP^2}(-\log D)\oplus \OO_{\PP^2}$ by the form $B$ as illustrated in Example \ref{global split log PA}, we obtain a split log Picard algebroid on $(\PP^2,D)$.
%
\end{Example}

For the remainder of this paper, all log Picard algebroids will be taken to be locally trivial log Picard algebroids.

\subsection{Residues and local systems}
We now describe a notion of residue for log Picard algebroids. We begin with the residue short exact sequence
\[
0\rightarrow\Omega_{X}^{1}\rightarrow\Omega_{X}^{1}(\log D)\overset{\mathrm{Res}}{\rightarrow}\oplus_{i=1}^{m}\mathcal{O}_{D_{i}}\rightarrow0
\]
where $D_{1},\dots,D_{m}$ are the irreducible components of the simple normal crossings divisor $D$. We can describe the residue map locally. Since $D$ is a simple normal crossings divisor, it can be expressed locally by the equation $f_{1}\dots f_{m}=0$ and as a result, the logarithmic 1-form can be written as
\[
\sum_{i=1}^{k}\alpha_{i}\frac{df_{i}}{f_{i}}+\gamma_{i}
\]
where $\alpha_{i},\gamma_{i}\in\mathcal{O}_{X}$. The residue map is given by
\[
\sum_{i=1}^{k}\alpha_{i}\frac{df_{i}}{f_{i}}+\gamma_{i}\overset{\mathrm{Res}}{\longmapsto}(\alpha_1,\dots,\alpha_k).
\]
If we consider closed differential forms, we have the following version of the residue short exact sequence
\[
0\rightarrow\Omega_{X}^{1,\mathrm{cl}}\rightarrow\Omega_{X}^{1,\mathrm{cl}}(\log D)\overset{\mathrm{Res}}{\rightarrow}\oplus_{i=1}^{k}\mathcal{\mathbb{C}}_{D_{i}}\rightarrow0
\]
with the following long exact sequence in cohomology
\begin{equation} \label{eq: residue seq for loc PA}
\dots\rightarrow H^{1}(\Omega_{X}^{1,\mathrm{cl}})\rightarrow H^{1}(\Omega_{X}^{1,\mathrm{cl}}(\log D))\overset{\mathrm{Res}_{*}}{\rightarrow}\oplus_{i=1}^{k}H^{1}(D_{i},\CC)\rightarrow\dots
\end{equation}
We explain the geometric meaning of the displayed segment of the long exact sequence as follows. 
Let $\wt\AA$ be a Picard algebroid with class in $H^{1}(\Omega_{X}^{1,\mathrm{cl}})$. The map to $H^{1}(\Omega_{X}^{1,\mathrm{cl}}(\log D))$ is simply the canonical log Picard algebroid $\AA$ associated to $\wt\AA$, obtained by pulling back the anchor map $\wt\sigma:\wt\AA\to \TT_{X}(-\log D)$ along the inclusion map $i:\TT_{X}(-\log D)\rightarrow\TT_{X}$. These maps are algebroid morphisms and so we obtain the following fiber product diagram of Lie algebroids.
\[
\xymatrix{
0\ar[r] &\mathcal{O}_{X}\ar[r] &\widetilde{\mathcal{A}}\ar[r]^{\widetilde{\sigma}}  &\mathcal{T}_{X}\ar[r] &0 \\
0\ar[r] &\mathcal{O}_{X}\ar[r] \ar@{=}[u] &\mathcal{A}\ar[r] \ar[u] &\mathcal{T}_{X}(-\log D)\ar[r] \ar[u]^i &0} 
\]
The residue map from $H^{1}(\Omega_{X}^{1,\mathrm{cl}}(\log D))$ to $\oplus_{i=1}^{k}H^{1}(D_{i},\CC)$ indicates that there is a $\CC$-local system on $D$ associated to any log Picard algebroid; indeed, if we restrict $\AA$ to $D$ (suppose $D$ is smooth and irreducible for simplicity) we obtain the sequence 
\[
0\rightarrow\mathcal{O}_{D}\rightarrow\mathcal{A}|_{D}\overset{\sigma|_{D}}{\rightarrow}\mathcal{T}_{X}(-\log D)|_{D}\rightarrow0.
\]
The sheaf $\mathcal{T}_{X}(-\log D)|_{D}$ is the Atiyah algebroid of the normal bundle $N_{D}$, so in particular, it is a Picard algebroid on $D$. Let $E$ be the central section of this Picard algebroid. We make the following definition:

\begin{Definition}	\label{residue of log PA def}
The \textbf{\textit{residue}} of the log Picard algebroid along $D$ is the $\mathbb{C}$-local system $R$ given by sections of the affine line bundle $\sigma|_D^{-1}(E)$ which are central in $\AA|_{D}$. 
\end{Definition}

\begin{Remark}
If we had worked with the Zariski topology rather than the analytic topology, the group $H^{1}(D,\CC)$ would vanish. Using the long exact sequence in \eqref{eq: residue seq for loc PA}, this would imply that all locally trivial log Picard algebroids on $(X,D)$ must be pullbacks of Picard algebroids on $X$. 
\end{Remark}

\subsection{Examples}
We will now consider examples of log Picard algebroids on $(X,D)$ where $D$ is a smooth irreducible divisor in $X$. First, we have the following proposition:

\begin{Proposition}	\label{log PA example prop}
Let $X$ be a projective algebraic surface with first Betti number $b_{1}(X)=0$ and $D$ a smooth irreducible curve in $X$.  The group of isomorphism classes of locally trivial log Picard algebroids on $(X,D)$ is then an extension of the group of $\CC$-local systems on $D$ by the group of isomorphism classes of algebraic log Picard algebroids on $(X,D)$.  
\end{Proposition}

\begin{proof}
First, we analyze the long exact sequence in \eqref{eq: residue seq for loc PA}
\begin{equation} \label{eq:residue sequence low degrees}
\dots\rightarrow H^{0}(D,\CC)\rightarrow H^{1}(\Omega_{X}^{1,\mathrm{cl}})\rightarrow H^{1}(\Omega_{X}^{1,\mathrm{cl}}(\log D))\overset{\mathrm{Res}_{*}}{\rightarrow}H^{1}(D,\CC)\rightarrow H^{2}(\Omega_{X}^{1,\mathrm{cl}})\rightarrow\dots
\end{equation}
Since $X$ is projective, we have the quasi-isomorphism of complexes
\[
\Omega_{X}^{1,\mathrm{cl}}\hookrightarrow[\Omega_{X}^{1}\rightarrow\Omega_{X}^{2}]
\]
which implies
\[
H^{2}(\Omega_{X}^{1,\mathrm{cl}})\cong\mathbb{H}^{2}(\Omega_{X}^{1}\rightarrow\Omega_{X}^{2})
\]
Moreover, the Hodge-de Rham spectral sequence degenerates at $E_{1}$ and so we have the short exact sequence
\[
0\rightarrow H^{1}(\Omega_{X}^{2})\rightarrow\mathbb{H}^{2}(\Omega_{X}^{1}\rightarrow\Omega_{X}^{2})\rightarrow H^{2}(\Omega_{X}^{1})\rightarrow0
\]
From the hypothesis, we have $b_{1}(X)=0$. By Poincare duality, this gives $b_{3}(X)=0$ and so the Hodge numbers $h^{1,2}$ and $h^{2,1}$ must be zero. Thus, $H^{2}(\Omega_{X}^{1,\mathrm{cl}})=0$. We may now rewrite~\eqref{eq:residue sequence low degrees} as follows
\[
0\rightarrow\mathrm{coker}(H^0(D,\CC)\rightarrow H^1(\Omega_X^{1,\mathrm{cl}}))\rightarrow H^1(\Omega_X^{1,\mathrm{cl}}(\log D))\rightarrow H^1(D,\CC)\rightarrow0. 
\]
It remains now to prove that the group of isomorphism classes of algebraic log Picard algebroids on $(X,D)$ is isomorphic to the cokernel of the map $H^0(D,\CC)\rightarrow H^1(\Omega_X^{1,\mathrm{cl}})$. To see this, consider the residue short exact sequence in the Zariski topology
\[
\dots\rightarrow H^0(D_\mathrm{alg},\CC)\rightarrow H^1(\Omega_{X,{\mathrm{alg}}}^{1,\mathrm{cl}})\rightarrow H^1(\Omega_{X,\mathrm{alg}}^{1,\mathrm{cl}}(\log D_\mathrm{alg}))\rightarrow H^1(D_\mathrm{alg},\CC)\rightarrow\dots
\]
As $D$ was assumed to be a smooth irreducible divisor, the map from $H^0(D_\mathrm{alg},\CC)$ to $H^1(\Omega_{X,{\mathrm{alg}}}^{1,\mathrm{cl}})$ is simply the degree map and hence injective. Furthermore, in the Zariski topology, the cohomology group $H^1(D_\mathrm{alg},\CC)$ is zero. Thus, we obtain that the group of isomorphism classes of algebraic log Picard algebroids $H^1(\Omega_{X,\mathrm{alg}}^{1,\mathrm{cl}}(\log D_\mathrm{alg}))$ is isomorphic to the cokernel of $H^0(D_\mathrm{alg},\CC)\rightarrow H^1(\Omega_{X,{\mathrm{alg}}}^{1,\mathrm{cl}})$. However, since $D$ and $X$ are both projective, by the GAGA principle, this implies that $H^1(\Omega_{X,\mathrm{alg}}^{1,\mathrm{cl}}(\log D_\mathrm{alg}))$ is isomorphic to the cokernel of the map $H^0(D,\CC)\rightarrow H^1(\Omega_X^{1,\mathrm{cl}})$ which is exactly what we wanted to show.
\end{proof}

We have an obvious source of examples of log Picard algebroids which are obtained by the pullback of Picard algebroids on $X$ as described earlier. We will now provide a class of examples of log Picard algebroids on $(X,D)$ which do not arise in this obvious fashion. 


\begin{Example} \label{curve in P2 PA}Consider $(\mathbb{P}^{2},D)$ where $D$ is a smooth projective curve of degree $d\geq3$ in $\mathbb{P}^{2}$. The map $H^0(D,\CC)\rightarrow H^1(\Omega_{\PP^2}^{1,\mathrm{cl}})$ is an isomorphism, hence there are no non-trivial algebraic log Picard algebroids for this example. Using the result of the above proposition, we obtain that
\[
H^{1}(\Omega_{\mathbb{P}^{2}}^{1,\mathrm{cl}}(\log D))\cong H^1(D,\CC) = \mathbb{C}^{2g},
\]
where $g$ is the genus of the curve $D$. The log Picard algebroids here are in one-to-one correspondence with the $\CC$-local systems on $D$.
\end{Example}

%

\subsection{Blowups and functoriality}
In this subsection, we describe the functoriality properties of log Picard algebroids, generalizing the treatment for usual Picard algebroids in Section 2.2 of \cite{BeilinsonBernstein93}. Let $\varphi:Y\rightarrow X$ be a morphism of projective varieties and $D$ a simple normal crossings divisor on $X$ such that $\varphi^{*}D$ is also a simple normal crossings divisor on $Y$. Suppose that $\mathcal{A}$ is a log Picard algebroid on $(X,D)$, consider the $\OO_{Y}$-module $\varphi^{!}\mathcal{A}:=\varphi^{*}\mathcal{A}\times_{\varphi^{*}\tlogd}\TT_{Y}(-\log \varphi^{*}D)$ with the Lie bracket on its sections defined as
\[
[(\xi,\sum_{i}f_{i}\otimes P_{i}),(\eta,\sum_{j}g_{j}\otimes Q_{j})]=([\xi,\eta],\sum_{i,j}f_{i}g_{j}\otimes[P_{i},Q_{j}]+\xi(g_j)\otimes P_{i}-\eta(f_i)\otimes Q_{j})
\]
where $\xi,\eta\in\TT_{Y}(-\log \varphi^{*}D)$, $f_i,g_j\in\OO_{Y}$ and $P_i,Q_j\in\varphi^{-1}\mathcal{A}$. The projection map $\pi_{\TT_{Y}(-\log \varphi^{*}D)}:\varphi^{!}\mathcal{A}\rightarrow\TT_{Y}(-\log \varphi^{*}D)$ preserves the Lie bracket and the central section $e_{\varphi^{!}\mathcal{A}}:\OO_{Y}\rightarrow\varphi^{!}\mathcal{A}$ is defined as $e_{\varphi^{!}\mathcal{A}}=(\varphi^{*}e_{\mathcal{A}},0)$. We make the following definition:

\begin{Definition}
The quadruple $(\varphi^{!}\mathcal{A},[\cdot,\cdot],\pi_{\TT_{Y}(-\log \varphi^{*}D)},e_{\varphi^{!}\mathcal{A}})$ is a log Picard algebroid on $(Y,\varphi^{*}D)$ and we call this the \textit{\textbf{pullback log Picard algebroid}} of $\mathcal{A}$. As before, we simply denote this quadruple by $\varphi^{!}\mathcal{A}$.
\end{Definition}

\begin{Theorem}	\label{log PA blowup thm} 
Let $X$ be a smooth projective algebraic surface $X$ with $b_{1}(X)=0$, $D$ a simple normal crossing divisor on $X$, and $\pi:\widetilde{X}\rightarrow X$ the blowup of $X$ at a point $p$ where $p\in D\subset X$. The pullback operation for log Picard algebroids defines an equivalence between the locally trivial log Picard algebroids on $(X,D)$ and the locally trivial log Picard algebroids on $(\widetilde{X},\pi^{*}D)$ where $\pi^{*}D:=D+E$ and $E$ is the exceptional divisor of the blowup.
\end{Theorem}

\begin{proof}
Let $\mathcal{A}$ be a locally trivial log Picard algebroid on $(X,D)$, the assignment $\mathcal{A}\mapsto\pi^{!}\mathcal{A}$ induces the map in cohomology
\[
H^{1}(\Omega_{X}^{1,\mathrm{cl}}(\log D))\rightarrow H^{1}(\Omega_{\widetilde{X}}^{1,\mathrm{cl}}(\log \pi^{*}D))
\]
and we want to show that this map is an isomorphism. From the proof of Theorem \ref{loc trivial log PA Hodge thm}, we have 
\begin{equation} \label{eq:Hodge weight for blowup}
H^{1}(\Omega_{\widetilde{X}}^{1,\mathrm{cl}}(\log \pi^{*}D))=F^{1}H^{2}(U,\CC)\cap W_{3}H^{2}(U,\CC)
\end{equation}
where the Hodge filtration and weight filtration on $H^{2}(U,\CC)$ is given by the compactification of $U$ into $\widetilde{X}$ by $D+E$. In order to prove our theorem, we now check that this indeed coincides with the Hodge and weight filtration on $H^{2}(U,\CC)$ for the usual compactification of $U$ into $X$ by $D$. First, we have
\[
F^{1}H^{2}(U,\CC)=\frac{H^{2}(U,\CC)}{H^{2}(\OO_{\widetilde{X}})}
\]
If we use the fact that the plurigenus is a birational invariant along with Serre duality, we have $H^{2}(\OO_{\widetilde{X}})=H^{2}(\OO_{X})$. This shows that $F^{1}H^{2}(U,\CC)$ does indeed coincide with the first Hodge filtration on $H^{2}(U,\CC)$ for the usual compactification of $U$ into $X$ by $D$.

As the weight filtration is independent of the choice of simple normal crossings divisor used in the compactification~\cite[3.2.11]{Deligne71}, $W_3H^2(U,\CC)$ for the compactification of $U$ into $\widetilde{X}$ by $D+E$ is the same as the compactification of $U$ into $X$ by $D$. This demonstrates that 
\[
H^{1}(\Omega_{X}^{1,\mathrm{cl}}(\log D))\rightarrow H^{1}(\Omega_{\widetilde{X}}^{1,\mathrm{cl}}(\log \pi^{*}D))
\]
is indeed an isomorphism. 
\end{proof}

\section{Meromorphic line bundles}	\label{sec 3}


In this section, we will work in the analytic topology unless otherwise specified. First, we consider the sheaf of algebras $\mathcal{O}_{X}(*D)$ of meromorphic functions on $X$ with poles of arbitrary order on $D$. There is a local description of this sheaf: if the simple normal crossings divisor $D$ is given by $f_{1}f_{2}\dots f_{m}=0$, we may write $\mathcal{O}_{X}(*D)=\mathcal{O}_{X}[f_{1}^{-1},\dots,f_{m}^{-1}]$. The sheaf of algebras $\OO_{X}(*D)$ has a $\ZZ$-filtration $F^{k}$ along $D$ defined as:
\begin{equation}	\label{eq:Z-filtration}
F^{k}\OO_{X}(*D):=\OO_{X}(kD),
\end{equation}
where $\OO_{X}(kD)$ is the sheaf of meromorphic functions on $X$ with poles bounded by $kD$. 
\begin{Definition}
A meromorphic line bundle $\mathcal{M}$ on $(X,D)$ is a locally free rank one $\OO_{X}(*D)$-module.
\end{Definition}
Just as for usual holomorphic line bundles, we may choose local trivializations
\begin{equation}\label{loctrm}
e_{i}:\mathcal{M}|_{U_i}\overset{\cong}{\rightarrow}\OO_{X}(*D)|_{U_i}
\end{equation}
of $\MM$ over an analytic open cover $\mathcal{U}=(U_i)_{i\in I}$ of $X$, so that $\MM$ is represented by the \v{C}ech 1-cocycle $(g_{ij})\in \check{Z}^1(\mathcal{U},\OO_{X}^{\times}(*D))$ defined by the requirement $e_j = e_i g_{ij}$ over $U_i\cap U_j$.  Here 
$\OO_{X}^{\times}(*D)$ denotes the sheaf of invertible elements of $\OO_{X}(*D)$, which may be described locally by adjoining group generators $f_{1},\dots,f_{m}$ to the sheaf $\mathcal{O}_{X}^{\times}$ of nowhere-vanishing holomorphic functions.  This means that a section of $\mathcal{O}_{X}^{\times}(*D)$ has the local description
\[
hf_{1}^{k_{1}}f_{2}^{k_{2}}\dots f_{m}^{k_{m}}, 
\]
with $h\in\mathcal{O}_{X}^{\times}$ and $k_{1},\dots,k_{m}\in\mathbb{Z}$. 
As a result, we have the short exact sequence of sheaves of abelian groups
\begin{equation}	\label{eq:O*(D) ses}
1\rightarrow\mathcal{O}_{X}^{\times}\rightarrow\mathcal{O}_{X}^{\times}(*D)\overset{\mathrm{ord}}{\rightarrow}\oplus_{i=1}^{m}\mathbb{Z}_{D_{i}}\rightarrow1,
\end{equation}
where the order map is defined by
\begin{equation}	\label{eq:ord map}
hf_{1}^{k_{1}}f_{2}^{k_{2}}\dots f_{m}^{k_{m}}\overset{\mathrm{ord}}{\mapsto}(k_1,\ldots, k_m).\end{equation}   

Using the order map, we may associate to any meromorphic line bundle $\MM$ a $\ZZ$-local system $\ord(\MM)$ along $D$, defined as follows.  Given local trivializations of $\MM$ as in~\eqref{loctrm}, observe that the sheaf $\OO_{X}(*D)|_{U_i}$ has a $\ZZ$-filtration along $D_{i}:=D\cap U_{i}$ as defined in \eqref{eq:Z-filtration}; let $\ZZ_{D_i}$ be the locally constant sheaf on $D_i$ which labels this filtration. Since, on the double overlap $U_{i}\cap U_j$, the transition function $g_{ij}\in\OO_{X}^{\times}(*D)(U_{ij})$  has order $k_{ij} = \ord(g_{ij})$, we see that the $k_i$-filtered part of $\OO_{X}(*D)|_{U_i}$ is identified with the $k_j = k_i + k_{ij}$-filtered part of $\OO_X(*D)|_{U_j}$.
We conclude from this that $\MM$, like $\OO_X(*D)$, is filtered along $D$, but by a $\ZZ$-local system given by gluing $\ZZ|_{D_i}$ to $\ZZ|_{D_j}$ using the translation $k_{ij} = \ord(g_{ij})$.  We denote this residual local system by $\ord(\MM)$ and summarize our findings below.
%

\begin{Theorem}	\label{mero picard group theorem} The \textbf{meromorphic Picard group} $\mathrm{Pic}(X,D)$ of isomorphism classes of meromorphic line bundles on $(X,D)$ is given by the cohomology group
\[
H^{1}(\mathcal{O}_{X}^{\times}(*D)), 
\]
and the residual $\ZZ$-local system $\ord(\MM)$ of a meromorphic line bundle $\MM$, as defined above, is classified by the map $\ord_*$ in the sequence
\begin{equation}	\label{eq: mero line bundle sequence}
\dots\rightarrow H^{0}(D,\ZZ)\rightarrow\mathrm{Pic}(X)\rightarrow H^{1}(\mathcal{O}_{X}^{\times}(*D))\overset{\mathrm{ord}_{*}}{\rightarrow} H^{1}(D,\ZZ)\rightarrow H^{2}(\mathcal{O}_{X}^{\times})\rightarrow\dots
\end{equation}
%
induced by the short exact sequence~\eqref{eq:O*(D) ses}.
\end{Theorem}

\begin{Remark} 
Note that the analytic topology is crucial for the existence of interesting meromorphic line bundles: if we work with the Zariski topology, then there are no nontrivial $\ZZ$-local systems on $D$, and so all meromorphic line bundles come from holomorphic line bundles on $X$.
\end{Remark}	

\begin{Theorem}	\label{mero picard extension thm} 
If the analytic Brauer group $H^{2}(\mathcal{O}_{X}^{\times})$ is trivial, then the meromorphic Picard group $\mathrm{Pic}(X,D)$ is an extension of the group of $\mathbb{Z}$-local systems on $D$ by the algebraic Picard group of $U$.\end{Theorem}

\begin{proof}
We consider the following short exact sequence in the Zariski topology
\[
0\rightarrow\mathcal{O}_{X,\mathrm{alg}}^{\times}\rightarrow j_{*}\mathcal{O}_{U,\mathrm{alg}}^{\times}\rightarrow i_{*}\mathbb{Z}_{D,\mathrm{alg}}\rightarrow0
\]
where $j:U\hookrightarrow X$ and $i:D\hookrightarrow X$. The associated long exact sequence is:
\begin{equation}	\label{eq: algebraic Picard seq}
\dots\rightarrow H^{0}(D,\ZZ)\rightarrow\mathrm{Pic}(X)\rightarrow\mathrm{Pic}^{\mathrm{alg}}(U)\rightarrow H^{1}(D,\ZZ)\rightarrow\dots
\end{equation}
The cohomology group $H^{1}(D,\ZZ)$ is zero since we are working in the Zariski topology here. Therefore, the algebraic Picard group $\mathrm{Pic}^{\mathrm{alg}}(U)$ is the cokernel of the map from $H^{0}(D,\mathbb{Z})$ to $\mathrm{Pic}(X)$. Since both $D$ and $X$ satisfy the GAGA principle, we can rewrite the long exact sequence in \eqref{eq: mero line bundle sequence} as:
\begin{equation}	\label{eq: mero line bundle sequence II}
0\rightarrow\mathrm{Pic}^{\mathrm{alg}}(U)\rightarrow\mathrm{Pic}(X,D)\rightarrow H^{1}(D,\ZZ)\rightarrow H^{2}(\mathcal{O}_{X}^{\times})
\end{equation}
If the analytic Brauer group $H^{2}(\mathcal{O}_{X}^{\times})$ is trivial, we obtain the required extension:
\[
0\rightarrow\mathrm{Pic}^{\mathrm{alg}}(U)\rightarrow\mathrm{Pic}(X,D)\rightarrow H^{1}(D,\ZZ)\rightarrow0.
\]
\end{proof}
%
\begin{Example}	\label{curve in P2 mero}
Consider $(\mathbb{P}^{2},D)$ where $D$ is a smooth irreducible projective curve of degree $d\geq3$ as in Example \ref{curve in P2 PA}. The group of $\mathbb{Z}$-local systems on $D$ is $\mathbb{Z}^{2g}$ where $g$ is the genus of $D$. Moreover, the algebraic Picard group of the complement $U$ is given by:
\[
\mathrm{Pic}^{\mathrm{alg}}(U)=\mathbb{Z}_{d}
\]
From Theorem \ref{mero picard extension thm}, we obtain:
\[
\mathrm{Pic}(\mathbb{P}^{2},D)\cong\mathbb{Z}_{d}\oplus\mathbb{Z}^{2g}
\]
In Section~\ref{sec 4}, we will give an explicit construction of these meromorphic line bundles for the case where $D$ is a smooth cubic curve in $\PP^2$.
\end{Example}

%
%
%

\subsection{The total space of a meromorphic line bundle} \label{local description subsec}
We now provide a geometric description of the total space of a meromorphic line bundle for the case where $D$ is a smooth divisor on $X$. First, we describe how this is done for the case where $X=\CC$ and $D=\left\{0\right\}$. 

Consider a family $(\CC^2_i)_{i\in\ZZ}$ of affine planes indexed by the integers and let $(x_i,y_i)$ be affine coordinates on $\CC^2_i$.  Let $\sim$ be the equivalence relation on $\coprod_{i\in \ZZ} \CC^2_i$ generated by the identification: 
\begin{equation}
(x_{i+1},y_{i+1}) = (y_i^{-1}, y_i^2 x_i),\ \ \text{for}\ y_i\neq 0.
\end{equation}
Quotienting by this relation, we obtain a smooth complex surface $M$:
\begin{equation}
M_0 :=  \left.\coprod_{i\in \ZZ} \CC^2_i \middle/\sim\right. .
\end{equation}
Since the map $(x_i,y_i)\mapsto x_iy_i$ is preserved by the equivalence relation, it defines a surjective holomorphic map $\pi_0:M\to\CC$.  We also have a natural $\CC^*$ action on $M_0$ preserving the projection $\pi_0$, defined by:
\begin{equation}
\lambda \cdot (x_i,y_i) = (\lambda x_i, \lambda^{-1}y_i).
\end{equation}

The fibre of $\pi_0$ over any nonzero point $z\in \CC$ may be identified with the hyperbola $\{(x_i,y_i)\in\CC^2_i \ :\ x_iy_i = z\}$ for any fixed $i$, so that the restriction of $(M_0,\pi_0)$ to $\CC\setminus\{0\}$ defines a (trivial) principal $\CC^*$ bundle.  On the other hand, the fibre of $\pi_0$ over $0\in\CC$ is 
\begin{equation}
\pi_0^{-1}(0)= \left.\coprod_{i\in\ZZ} \{(x_i,y_i)\ :\ x_iy_i = 0\} \middle/ \sim,  \right.
\end{equation}
which is an infinite chain of rational curves with simple normal crossings, depicted below. 

\begin{figure}[h]
\centering
\begin{tikzpicture}
\draw (0,0) arc (40:140:1cm);
\draw (1.3,0) arc (40:140:1cm);
\draw (-1.3,0) arc (40:140:1cm);
\draw (2.6,0) arc (40:140:1cm);
\draw (-2.6,0) arc (40:140:1cm);
\node at (3,.1) {$\cdots$};
\node at (-4.5,.1) {$\cdots$};
\end{tikzpicture}
\end{figure}

The singular points of the above chain coincide with the fixed points of the $\CC^*$ action on $M_0$, given by the collection $(x_i)_{i\in\ZZ}$ where $x_i = (0,0)\in\CC^2_i$.  

\begin{Theorem}
Let $\pi_0:M_0\to \CC$ be the holomorphic map defined above.  Then the sheaf of automorphisms of $(M_0,\pi_0)$ respecting the $\CC^{*}$-action on $M_0$ is given by $\OO_{X}^{\times}(*D)$, where $(X,D) = (\CC,\{0\})$.
\end{Theorem}


\begin{proof}
Away from $\pi_0^{-1}(0)$, an automorphism $\phi$ is a principal bundle automorphism, and so is given in any of the above local charts $\CC^2_i$ by 
\[
\phi(x_i,y_i) = (g(z)x_i, g(z)^{-1} y_i),
\]
where $z=\pi_0(x_i,y_i) = x_iy_i$ and $g$ is a nonvanishing holomorphic function on $\CC\backslash\{0\}$. We must now investigate the behaviour of $g(z)$ as $z$ approaches zero.

The fixed point $(x_0,y_0)=(0,0)$ must be sent to another one by the automorphism $\phi$, say $(x_k,y_k)=(0,0)$. So the automorphism defines a germ of a holomorphic map sending the origin of $\CC^2_0$ to the origin of $\CC^2_k$, meaning that
\[
\phi(x_0,y_0) = (\alpha(x_0,y_0), \beta(x_0,y_0)) 
\]  
for $(\alpha,\beta)$ holomorphic and with nonzero Jacobian at zero. 

The equivalence relation $\sim$ identifies $(x_0,y_0)$ with 
$(x_k,y_k) =(y_0^{-k}x_0^{-k+1},y_0^{k+1}x_0^k) = (z^{-k} x_0, z^{k} y_0)$,
 and so takes $(g(z) x_0, g(z)^{-1}y_0)$ to $(z^{-k} g(z) x_0, z^{k} g(z)^{-1} y_0)$ in the $\CC^2_k$ chart.  As a result we have
\[
(\alpha(x_0,y_0), \beta(x_0,y_0)) = (z^{-k} g(z) x_0, z^{k} g(z)^{-1} y_0).
\]
This means that $z^{-k} g(z)$ extends to a holomorphic nonvanishing function in a neighbourhood of zero, that is, $g(z) = e^f z^{k}$ for $f$ holomorphic, showing that $g\in \mathcal{O}_{X}^\times(*D)$ as required. 
\end{proof}

\begin{Corollary}
For a smooth divisor $D\subset X$, a meromorphic $\CC^*$ principal bundle on $(X,D)$ may be defined as a complex manifold $M$ equipped with a surjective holomorphic map $\pi:M\to X$ and a $\CC^*$ action preserving $\pi$, such that over $X\backslash D$ it is a principal $\CC^*$ bundle, and near a point $p\in D$, it has the local normal form $(M_0,\pi_0)$ given above. 
\end{Corollary}

To be more precise, near a point $p\in D$, let $U$ be a neighbourhood of $p$ in which $(U,D\cap U)\cong (\D^{n-1}\times \D, \D^{n-1}\times\{0\})$, for $\D\subset \CC$ the open unit disc, and let $q:U\to \D$ be the second projection.  Then $(M,\pi)$ is isomorphic to $(q^*M_0, q^*\pi_0)$ when restricted to $U$.

Any usual principal $\CC^*$-bundle $P$ gives rise to a meromorphic principal bundle in the following way. Let $L$ be the complex line bundle associated to $P$, and let $\overline{L} = \mathbb{P}(L\oplus \OO_X)$ be its completion to a projective line bundle.   Then let $S = S_0 + S_\infty$ be the union of the zero and infinity sections of $\overline{L}$.   We construct a family of varieties by iterated blow-up: first blow up the intersection of $S$ with the restriction $L_D$ of $L$ to $D$, each  viewed as a submanifold of the total space of $\overline{L}$.  By iteratively blowing up the intersections of the proper transform of $S$ with the exceptional divisor, we obtain a sequence $(\widetilde M_k)_{k\geq 0}$ of varieties.  Deleting the proper transform of $S$ from each $\widetilde M_k$, we obtain a sequence $(M_k)_{k\geq 0}$ of varieties equipped with inclusion maps $M_k\hookrightarrow M_{k+1}$.  Taking the direct limit of this system of varieties, we obtain a complex manifold (non-finite-type scheme) which is naturally a meromorphic principal bundle as defined above.

The functor described above, which maps usual principal bundles to meromorphic ones, categorifies the natural map 
\[
H^1(X,\mathcal{O}_{X}^\times)\to H^1(X,\mathcal{O}_{X}^\times(*D))
\]
which derives from the natural inclusion of automorphism sheaves.

\subsection{Integral mixed Hodge theory}
In Section \ref{sec 2}, we introduced the theory of weight filtrations in order to obtain classification results for log Picard algebroids. Here, we will introduce a modification of Deligne's weight filtration for integral coefficients to obtain classification results for meromorphic line bundles. 

Let $j:U\hookrightarrow X$ be the inclusion of the complement $U:=X\backslash D$ into $X$ and consider the integral cohomology group $H^{k}(U,\ZZ)$. This can be computed using the complex $Rj_{*}\ZZ_{U}$:
\[
H^{k}(U,\ZZ)=\mathbb{H}^{k}(Rj_{*}\ZZ_{U})
\]
We now define the integral weight filtration on $H^{k}(U,\ZZ)$ by using the canonical filtration $\sigma$ introduced earlier in Section \ref{sec 2} on the complex $Rj_{*}\mathbb{Z}_{U}$. We define the $\textit{\textbf{integral weights}}$ $W_{m}H^{k}(U,\mathbb{Z})$ as:
\begin{equation}	\label{eq:integral weights}
W_{m}H^{k}(U,\mathbb{Z}):=\mathrm{im}(\mathbb{H}^{k}(\sigma_{\leq m-k}Rj_{*}\mathbb{Z}_{U})\rightarrow H^{k}(U,\mathbb{Z}))
\end{equation}
This gives the $\textit{\textbf{integral weight filtration}}$ on $H^{k}(U,\mathbb{Z})$:
\[
H^{k}(U,\mathbb{Z})=W_{2k}\supset W_{2k-1}\supset\dots\supset W_{k}\supset\{0\}
\]
Consider now the Leray spectral sequence with respect to the canonical filtration $\sigma$:
\begin{equation}	\label{eq:integral Leray I}
E_{2}^{k-m,m}=H^{k-m}(R^{m}j_{*}\mathbb{Z}_{U})\Rightarrow H^{k}(U,\mathbb{Z})
\end{equation}
By \cite[Lemma 4.9]{PetersSteenbrink08}, we have:
\begin{equation}	\label{eq:RM=ZD}
R^{m}j_{*}\mathbb{Z}_{U}=\alpha_{m*}\mathbb{Z}_{D^{(m)}}(-m)
\end{equation}
where $D^{(m)}$ is the disjoint union of all $m$-fold intersections of the different irreducible components of the divisor $D$ as before, $\alpha_{m}:D^{(m)}\hookrightarrow X$ is the natural inclusion, and $\mathbb{Z}_{D^{(m)}}(-m)$ on $D^{(m)}$ is the Tate twist of $\ZZ$ by $(2\pi i)^{-m}$.
Hence, this allows us to re-write the Leray spectral sequence in \eqref{eq:integral Leray I} as:
\[
E_{2}^{k-m,m}=H^{k-m}(D^{(m)},\mathbb{Z})(-m)\Rightarrow H^{k}(U,\mathbb{Z})
\]
This may be viewed as the integral analogue of the weight spectral sequence in \eqref{eq: weight spectral sequence}.

We now apply the theory of integral weight filtrations to study meromorphic line bundles on $(X,D)$. From \cite[2-6]{Brylinski94}, there is an exact sequence of sheaves
\[
0\rightarrow j_{*}\mathbb{Z}(1)_{U}\rightarrow\mathcal{O}_{X}\overset{\exp}{\rightarrow}\mathcal{O}_{X}^{\times}(*D)\rightarrow R^{1}j_{*}\mathbb{Z}(1)_{U}\rightarrow0
\]
where $\mathbb{Z}(1)$ is the Tate twist of $\ZZ$ by $2\pi i$. This implies that the sheaf $\mathcal{O}_{X}^{\times}(*D)$ is quasi-isomorphic to the mapping cone:
\begin{equation}	\label{eq: cone for OX(D)}
\mathcal{O}_{X}^{\times}(*D)\sim\mathrm{Cone}(\sigma_{\leq1}\mathbb{R}j_{*}\mathbb{Z}(1)_{U}\rightarrow\mathcal{O}_{X})
\end{equation}
We can write the long exact sequence associated to this mapping cone:
\begin{equation}	\label{eq:cone for OX(D) les}
\begin{split}
\dots & \rightarrow\mathbb{H}^{1}(\sigma_{\leq 1}Rj_{*}\ZZ(1)_{U})\rightarrow H^{1}(\OO_{X})\overset{\exp}{\rightarrow}\mathrm{Pic}(X,D) \\
& \rightarrow\mathbb{H}^{2}(\sigma_{\leq 1}Rj_{*}\ZZ(1)_{U})\rightarrow H^{2}(\OO_{X})\rightarrow\dots
\end{split}
\end{equation}
We would like to make a couple of observations here. From the definition of integral weights \eqref{eq:integral weights}, we have:
\[
W_{3}H^{2}(U,\ZZ(1)):=\mathrm{im}(\mathbb{H}^{2}(\sigma_{\leq 1}Rj_{*}\mathbb{Z}(1)_{U})\rightarrow H^{2}(U,\ZZ(1)))
\]
If we use the following natural short exact sequence for the canonical filtration $\sigma$:
\begin{equation}	\label{eq:canonical filtration ses}
0\rightarrow\sigma_{\leq 1}Rj_{*}\mathbb{Z}(1)_{U}\rightarrow Rj_{*}\mathbb{Z}(1)_{U}\rightarrow\sigma_{>1}Rj_{*}\ZZ(1)_{U}\rightarrow0
\end{equation}
with its associated long exact sequence in hypercohomology
\begin{equation}	\label{eq:W3 into H2}
\dots\rightarrow\mathbb{H}^{1}(\sigma_{>1}Rj_{*}\mathbb{Z}(1)_{U})\rightarrow\mathbb{H}^{2}(\sigma_{\leq 1}Rj_{*}\ZZ(1)_{U})\rightarrow H^{2}(U,\ZZ(1))\rightarrow\dots
\end{equation}
Since $\mathbb{H}^{1}(\sigma_{>1}Rj_{*}\mathbb{Z}(1)_{U})=0$, this shows that the map
\[
\mathbb{H}^{2}(\sigma_{\leq 1}Rj_{*}\ZZ(1)_{U})\rightarrow H^{2}(U,\ZZ(1))
\]
is necessarily injective and so the weight 3 piece $W_{3}H^{2}(U,\ZZ(1))$ can be naturally identified with $\mathbb{H}^{2}(\sigma_{\leq 1}Rj_{*}\mathbb{Z}(1)_{U})$:
\begin{equation} \label{eq:W3 as hyperH2}
W_{3}H^{2}(U,\ZZ(1))\cong\mathbb{H}^{2}(\sigma_{\leq 1}Rj_{*}\mathbb{Z}(1)_{U})
\end{equation}
The second observation is that if we use the long exact sequence in hypercohomology associated to the short exact sequence in \eqref{eq:canonical filtration ses} again, we have:
\begin{equation*}
\begin{split}
\dots & \rightarrow\mathbb{H}^{0}(\sigma_{>1}Rj_{*}\mathbb{Z}(1)_{U})\rightarrow\mathbb{H}^{1}(\sigma_{\leq 1}Rj_{*}\ZZ(1)_{U})\rightarrow H^{1}(U,\ZZ(1)) \\
& \rightarrow\mathbb{H}^{1}(\sigma_{>1}Rj_{*}\mathbb{Z}(1)_{U})\rightarrow\dots
\end{split}
\end{equation*}
But $\mathbb{H}^{0}(\sigma_{>1}Rj_{*}\mathbb{Z}(1)_{U})$ and $\mathbb{H}^{1}(\sigma_{>1}Rj_{*}\mathbb{Z}(1)_{U})$ are zero so this implies that we have an isomorphism:
\[
\mathbb{H}^{1}(\sigma_{\leq 1}Rj_{*}\ZZ(1)_{U})\cong H^{1}(U,\ZZ(1))
\]
These two observations allows us to rewrite the sequence in \eqref{eq:cone for OX(D) les} as:
\begin{equation}	\label{eq:cone for OX(D) les II}
\dots\rightarrow H^{1}(U,\ZZ(1))\rightarrow H^{1}(\OO_{X})\overset{\exp}{\rightarrow}\mathrm{Pic}(X,D)\rightarrow W_{3}H^{2}(U,\ZZ(1))\rightarrow H^{2}(\OO_{X})\rightarrow\dots
\end{equation}
We may think of this sequence as a meromorphic analogue of the usual exponential sheaf sequence on $X$.

\begin{Definition}
We define the \textbf{\textit{meromorphic N\'eron-Severi group}} $\mathrm{NS}(X,D)$ of $(X,D)$ as the following subgroup of $H^{2}(U,\ZZ(1))$:
\[
\mathrm{NS}(X,D) = F^{1}H^{2}(U,\CC)\times_{H^{2}(U,\CC)} W_{3}H^{2}(U,\ZZ(1))\subset H^{2}(U,\ZZ(1)).
\]
\end{Definition}

\begin{Definition}
We define the following quotient group
\[
\frac{H^{1}(\OO_{X})}{\text{im}(H^{1}(U,\ZZ(1))\rightarrow H^{1}(\OO_X))}
\]
to be the \textbf{\textit{meromorphic Jacobian}} $\mathrm{Jac}(X,D)$ of $(X,D)$. The map $H^{1}(U,\ZZ(1))\rightarrow H^{1}(\OO_X)$ is the map given in the sequence in \eqref{eq:cone for OX(D) les II}.
\end{Definition}

\begin{Proposition}
Suppose that $H^{1}(U,\ZZ(1))$ is pure of weight 1, i.e.
\[
W_{1}H^{1}(U,\ZZ(1))=H^{1}(U,\ZZ(1)),
\] 
then the meromorphic Jacobian $\mathrm{Jac}(X,D)$ is isomorphic to the usual Jacobian of $X$.
\end{Proposition}

\begin{proof}
It suffices to show that $W_{1}H^{1}(U,\ZZ(1))$ is equal to $H^{1}(X,\ZZ(1))$. By definition:
\[
W_{1}H^{1}(U,\ZZ(1)):=\mathrm{im}(\mathbb{H}^{1}(\sigma_{\leq 0}Rj_{*}\ZZ(1)_{U})\rightarrow H^{1}(U,\ZZ(1)))
\]
If we use the natural short exact sequence for the canonical filtration $\sigma$: 
\[
0\rightarrow\sigma_{\leq 0}Rj_{*}\mathbb{Z}(1)_{U}\rightarrow Rj_{*}\mathbb{Z}(1)_{U}\rightarrow\sigma_{>0}Rj_{*}\ZZ(1)_{U}\rightarrow0
\]
This gives the long exact sequence in hypercohomology:
\[
\dots\rightarrow\mathbb{H}^{0}(\sigma_{>0}Rj_{*}\ZZ(1)_{U})\rightarrow\mathbb{H}^{1}(\sigma_{\leq 0}Rj_{*}\mathbb{Z}(1)_{U})\rightarrow\mathbb{H}^{1}(U,\ZZ(1))\rightarrow\dots
\]
Since $\mathbb{H}^{0}(\sigma_{>0}Rj_{*}\ZZ(1)_{U})=0$, this implies that $W_{1}H^{1}(U,\ZZ(1))$ can be naturally identified with the hypercohomology group $\mathbb{H}^{1}(\sigma_{\leq 0}Rj_{*}\mathbb{Z}(1)_{U})$. Using the definition of the canonical filtration in addition with the statement in \eqref{eq:RM=ZD}, we have:
\[
\sigma_{\leq 0}Rj_{*}\mathbb{Z}(1)_{U}= R^{0}j_{*}\ZZ(1)_{U}= \ZZ(1)_{X}
\]
This shows that $W_{1}H^{1}(U,\ZZ(1))$ is equal to $H^{1}(X,\ZZ(1))$.
\end{proof}

We have the following classification theorem for the meromorphic Picard group $\mathrm{Pic}(X,D)$:

\begin{Theorem}
The first Chern class of meromorphic line bundles $\MM$ on $(X,D)$ restricted to the complement $U=X\setminus D$ expresses the meromorphic Picard group $\mathrm{Pic}(X,D)$ as an extension of the meromorphic N\'eron-Severi group $\mathrm{NS}(X,D)$ by the meromorphic Jacobian $\mathrm{Jac}(X,D)$:
\[
0\rightarrow\mathrm{Jac}(X,D)\rightarrow\mathrm{Pic}(X,D)\overset{c_{1}}{\rightarrow}\mathrm{NS}(X,D)\rightarrow0
\]
\end{Theorem}

\begin{proof}
Using the long exact sequence in \eqref{eq:cone for OX(D) les II}, we obtain:
\[
0\rightarrow\mathrm{im}(\mathrm{Pic}(X,D)\rightarrow W_{3}H^{2}(U,\ZZ(1)))\rightarrow W_{3}H^{2}(U,\ZZ(1))\rightarrow H^{2}(\OO_{X})\rightarrow\dots
\]
The map $W_{3}H^{2}(U,\ZZ(1))\rightarrow H^{2}(\OO_{X})$ factors through the following composition of maps:
\[
W_{3}H^{2}(U,\ZZ(1))\rightarrow H^{2}(U,\mathbb{C})\rightarrow H^{2}(\OO_{X})
\]
where the first map above is the composition of the inclusion of $W_{3}H^{2}(U,\ZZ(1))\hookrightarrow H^{2}(U,\ZZ(1))$ and the canonical map from $H^{2}(U,\ZZ(1))\rightarrow H^{2}(U,\mathbb{C})$. The second map $H^{2}(U,\mathbb{C})\rightarrow H^{2}(\OO_{X})$ results from taking the quotient of $H^{2}(U,\mathbb{C})$ by the first Hodge filtration $F^{1}H^{2}(U,\mathbb{C})$. Hence, we have a commutative diagram:
\[
\xymatrix{
& \mathrm{im}(\mathrm{Pic}(X,D)\rightarrow W_{3}H^{2}(U,\ZZ(1)))\ar[r]\ar[d] & W_{3}H^{2}(U,\ZZ(1))\ar[d] & & \\
0\ar[r] & F^{1}H^{2}(U,\mathbb{C})\ar[r] & H^{2}(U,\mathbb{C})\ar[r] & H^{2}(\OO_{X})\ar[r]& 0}
\]
Thus, the image of $\mathrm{Pic}(X,D)$ in $W_{3}H^{2}(U,\ZZ(1))$ is exactly the fibre product $F^{1}H^{2}(U,\CC)\times_{H^{2}(U,\CC)} W_{3}H^{2}(U,\ZZ(1))$ which is the definition of the meromorphic N\'eron-Severi group $\mathrm{NS}(X,D)$. Now, if we use the long exact sequence in \eqref{eq:cone for OX(D) les II} again, we find that the kernel of $\mathrm{Pic}(X,D)\rightarrow W_{3}H^{2}(U,\ZZ(1))$ is the group:
\[
\frac{H^{1}(\OO_{X})}{\mathrm{im}(H^{1}(U,\ZZ(1))\rightarrow H^{1}(\OO_{X})}
\]
which is exactly the definition of the meromorphic Jacobian $\mathrm{Jac}(X,D)$.
\end{proof}

We conclude this subsection by proving an equivalence between the meromorphic Picard group on $(X,D)$ and the analytic Picard group of the complement $U$ under specific conditions.

\begin{Corollary} \label{mero line bundle=H2(U,Z)}
Let $X$ be a smooth projective variety with vanishing Hodge numbers $h^{0,1}=0, h^{0,2}=0$ and $D$ a smooth ample divisor in $X$. There is an isomorphism between the meromorphic Picard group $\mathrm{Pic}(X,D)$ and the analytic Picard group $\mathrm{Pic}(U^{an})$ of the complement $U$:
\[
\mathrm{Pic}(X,D)\cong\mathrm{Pic}(U^{an})
\]
\end{Corollary}

\begin{proof}
The assumption $H^{1}(\OO_X)=0$ implies that the meromorphic Jacobian $\mathrm{Jac}(X,D)=0$. Moreover, the assumption $H^{2}(\OO_X)=0$ implies that $F^{1}H^{2}(U,\CC)=H^{2}(U,\CC)$ and so the meromorphic N\'eron-Severi group $\mathrm{NS}(X,D)=W_{3}H^{2}(U,\ZZ(1))$. Using the above theorem, we obtain the following isomorphism
\[
\mathrm{Pic}(X,D)\cong W_{3}H^{2}(U,\ZZ(1))
\] 
Now, if we use the exact sequence in \eqref{eq:W3 into H2}, we find that the cokernel of the map:
\begin{equation}	\label{eq:W3 into W4}
W_{3}H^{2}(U,\ZZ(1))\rightarrow H^{2}(U,\ZZ(1))
\end{equation}
is precisely the kernel of the following map:
\[
\mathbb{H}^{2}(\sigma_{>1}Rj_{*}\ZZ(1)_{U})\rightarrow\mathbb{H}^{3}(\sigma_{\leq 1}Rj_{*}\ZZ(1))
\]
However, the group $\mathbb{H}^{2}(\sigma_{>1}Rj_{*}\ZZ(1)_{U})=H^{0}(R^{2}j_{*}\ZZ(1)_{U})$ and since $D$ was assumed to be a smooth divisor, this group is necessarily zero. Hence, the cokernel of the map in \eqref{eq:W3 into W4} is zero and thus,
\[
\mathrm{Pic}(X,D)\cong H^{2}(U,\ZZ(1))
\]
To conclude the proof, we need to show that the analytic Picard group $\mathrm{Pic}(U^{an})$ is also isomorphic to $H^{2}(U,\ZZ(1))$. Consider the long exact sequence arising from the exponential sheaf sequence on $U$:
\[
\dots\rightarrow H^{1}(\mathcal{O}_{U})\rightarrow\mathrm{Pic}(U^{an})\rightarrow H^{2}(U,\mathbb{Z}(1))\rightarrow H^{2}(\mathcal{O}_{U})\rightarrow\dots
\]
Since $U$ is the complement of a smooth ample divisor in $X$, $U$ must be an affine variety. Therefore, the higher cohomologies of coherent sheaves necessarily vanish, yielding
\[
\mathrm{Pic}(U^{an})\cong H^{2}(U,\mathbb{Z}(1)).
\]
\end{proof}

\subsection{Geometric Prequantization}
In this section we explain how meromorphic line bundles provide the appropriate notion of prequantization for logarithmic Picard algebroids. 

\subsubsection{Prequantization of Picard algebroids}

The Atiyah algebroid $\AA_{\LL}$ of a holomorphic line bundle $\mathcal{L}$ over $X$ is the sheaf of first-order differential operators on sections of $\mathcal{L}$.  It has a Lie bracket given by the commutator of operators, and the symbol map $\xi\mapsto \sigma_{\xi}$ defines a bracket-preserving morphism $\AA_{\LL}\rightarrow\mathcal{T}_{X}$ to the tangent sheaf with kernel $\OO_{X}$ given by the differential operators of order zero. In this way, the Atiyah algebroid $\AA_{\LL}$ gives an example of a Picard algebroid on $X$. 

%
%
%

\begin{Proposition}	\label{dlog PA prop}
Let $\LL$ be a holomorphic line bundle and let $\AA_\LL$ be its associated Atiyah algebroid.  Then $\dlog_*[\LL] = [\AA_\LL]$, where $[\mathcal{L}]\in H^1(\OO_X^\times)$ and $[\AA_\LL]\in H^1(\Omega^{1,\mathrm{cl}}_X)$ are the corresponding isomorphism classes and $\dlog_*$ is induced by the homomorphism of sheaves of abelian groups
\[
\xymatrix{\OO_X^\times\ar[r]^-{\dlog}& \Omega_X^{1,\mathrm{cl}}}.
\]
%
%
\end{Proposition}

\begin{proof}
Choose local trivializations $(e_i)_{i\in I}$ for $\LL$ over an open affine cover $\mathcal{U}=(U_i)_{i\in I}$ of $X$, so that $e_j = e_i g_{ij}$ defines a \v{C}ech representative $(g_{ij})\in \check{{Z}}^1(\mathcal{U},\OO_X^\times)$ for the isomorphism class $[\LL]\in H^1(\OO_X^\times)$. Each local trivialization $e_i$ defines a local flat connection $\nabla_i$ by imposing $\nabla_i(e_i) = 0$, giving a local splitting of the sequence of algebroids
\[
0\rightarrow\OO_{X}\rightarrow\AA\rightarrow\mathcal{T}_{X}\rightarrow 0.
\]
Then over $U_i\cap U_j$, we have $\nabla_i - \nabla_j = A_{ij}$, defining a representative $(A_{ij})\in \check{{Z}}^1(\mathcal{U}, \Omega^{1,\mathrm{cl}}_X)$ for the isomorphism class $[\AA_\LL]$.  Applying $\nabla_j$ to $e_j$, we obtain
\[
0 = \nabla_j(e_j) = \nabla_j(e_i g_{ij}) =  (dg_{ij}) e_i + ((\nabla_i - A_{ij})e_i) g_{ij} = (dg_{ij}-A_{ij} g_{ij}) e_i,
\]
implying $A_{ij} = \dlog(g_{ij})$, as required. We omit the verification that a different choice of local trivialization affects the cocycle $(A_{ij})$ by a coboundary.  
\end{proof}

%
%

The prequantization problem is the following: under what conditions may a Picard algebroid $\AA$ on $X$ be prequantized, that is,  when does there exist a holomorphic line bundle $\LL$ on $X$ such that $\AA\cong\AA_{\LL}$?  We now review Weil's well-known answer to this question, which is the starting point for the theory of geometric quantization.

\begin{Theorem} \cite[Ch. V, no. 4, Lemme 2]{Weil58}	\label{PA prequantization theorem}
Let $\AA$ be a Picard algebroid on $X$. There exists a prequantization for $\AA$ if and only if it has integral periods, meaning that $[\AA]\in H^{1}(\Omega_{X}^{1,\text{cl}})\subset H^{2}(X,\CC)$ lies in the image of the map $H^{2}(X,\ZZ)\rightarrow H^{2}(X,\CC)$.
\end{Theorem}

\begin{Corollary}	\label{PA prequantization corollary}
Picard algebroids on $X$ admitting prequantization are classified by the group of $(1,1)$ classes with integer periods\footnote{By the Lefschetz theorem on $(1,1)$ classes, this is isomorphic to the N\'eron-Severi group modulo torsion}, defined by 
\[
H^{1,1}(X,\ZZ):= H^{1,1}(X)\cap\text{im}(H^{2}(X,\ZZ)\rightarrow H^{2}(X,\CC)).
\]
\end{Corollary}

\begin{proof}
Let $\AA$ be a Picard algebroid, with isomorphism class in $H^{1}(\Omega_{X}^{1,\text{cl}}) = H^{2,0}(X)\oplus H^{1,1}(X)$. If $\AA$ is prequantizable, then by the above theorem, $[\AA]$ must lie in the image of the map from $H^{2}(X,\ZZ)$ and is therefore real.  Therefore $[\AA]\in H^{1,1}(X)$, as required. 
%
\end{proof}


\begin{Proposition}The line bundles prequantizing a fixed prequantizable Picard algebroid form a torsor for the group of holomorphic line bundles with torsion first Chern class. 
\end{Proposition}

\begin{proof}
We make use of the short exact sequence 
\[
\xymatrix{0\ar[r] & \CC_{X}^{\times}\ar[r] & \OO_{X}^{\times}\ar[r]^-{\dlog}& \Omega_{X}^{1,\mathrm{cl}}\ar[r] & 0}
\]
and its associated long exact sequence
\[
0\rightarrow H^{0}(\Omega_{X}^{1,\text{cl}})\rightarrow H^{1}(X,\CC^{\times})\overset{\alpha}\rightarrow H^{1}(\OO_{X}^{\times})\overset{\dlog_*}{\rightarrow} H^{1}(\Omega_{X}^{1,\text{cl}})\rightarrow\dots.
\]
Prequantizations of $\AA$ are classified by lifts of $[\AA]$ under the above $\dlog_*$ map.  Lifts of the class $[\AA]$ therefore form a torsor for the image $\mathcal{K}$ of the map $\alpha$, which are precisely those holomorphic line bundles admitting flat connection.  These may be characterized as the line bundles $\LL$ with $c_1(\LL)$ torsion, which form an extension 
\[
0\rightarrow\text{Jac}(X)\rightarrow\mathcal{K}\rightarrow H^{2}(X,\ZZ)_{\text{tors}}\rightarrow0.
\]
\end{proof}

\subsubsection{Prequantization of log Picard algebroids}
In this section we generalize the results of the previous section to allow logarithmic singularities.  We begin by introducing the sheaf of filtered differential operators on $(X,D)$:

\begin{Definition}	
The \textbf{\textit{filtered differential operators}} $\DD_{X,D}$ on $(X,D)$ is the sheaf of differential operators on $\OO_{X}(*D)$ which preserve the $\ZZ$-filtration on $\OO_{X}(*D)$ given in \eqref{eq:Z-filtration}. 
\end{Definition}


\begin{Definition}
Let $\mathcal{M}$ be a meromorphic line bundle on $(X,D)$. A first-order filtered differential operator on $\mathcal{M}$ is a map $\xi:\mathcal{M}\rightarrow\mathcal{M}$ which preserves the $\ord(\MM)$-filtration on $\mathcal{M}$ and such that, for any $s\in\mathcal{M}$ and $f\in\OO_{X}(*D)$, the symbol $\sigma_\xi$, defined by 
\[
\sigma_{\xi}(f)s = \xi(fs)-f\xi(s),
\]
is a derivation of $\OO_{X}(*D)$. We denote this sheaf of operators by $\AA_\MM$ and call it the {\bf Atiyah algebroid of $\MM$}. 
\end{Definition}

\begin{Proposition}
The Atiyah algebroid $\AA_\MM$, with its natural commutator bracket and symbol map, defines a locally trivial log Picard algebroid on $(X,D)$.
\end{Proposition}

\begin{proof}
We first show that $\sigma_{\xi}\in\TT_{X}(-\log D)$ for $\xi\in\AA_\MM$. We assume for simplicity that the divisor $D$ is smooth and is given locally by $\{z=0\}$, with the general simple normal crossings case following from the same argument.  From the defining property of $\xi$,
\begin{equation*}
\begin{split}
[\xi, z^{-1}]s & = \xi(z^{-1}s) - z^{-1}\xi(s) \\ 
& = -z^{-2}\sigma_{\xi}(z)s,
\end{split}
\end{equation*}
for $s\in\mathcal{L}$. If $s$ is in the degree $k$ part of the $\mathrm{ord}(\mathcal{M})$-filtration defined on $\mathcal{M}$, the left-hand side of the above equation is in the degree $k+1$ part which implies that $\sigma_{\xi}(z)\subset (z)$. Thus, $\sigma_{\xi}\in\TT_{X}(-\log D)$.

Since $z\partial_z$ acts on $\CC[z,z^{-1}]$ preserving the order filtration, it follows that the symbol is a surjective map from $\AA_\MM$ to $\TT_{X}(-\log D)$. Its kernel consists of the 0-th order filtered differential operators, which is precisely $\OO_{X}$, acting by multiplication and generated by the central section $e=1$. Therefore, we have the short exact sequence
\begin{equation}\label{extlogam}
0\rightarrow\OO_{X}\rightarrow\AA_\MM\rightarrow\TT_{X}(-\log D)\rightarrow0.
\end{equation}
As in the usual case, the symbol is bracket-preserving, completing the verification that $\AA_\MM$ is a log Picard algebroid on $(X,D)$. The fact that $\AA_\MM$ is locally trivial follows directly from the fact that $\MM$ is: if $e\in \MM|_{U}$ is a local trivialization, then the condition $\nabla e = 0$ defines a unique meromorphic connection on $\MM_{U}$, and hence gives a local splitting of~\eqref{extlogam}.
\end{proof}

Having established the above, we obtain (using the same proof) the following analogue of Proposition \ref{dlog PA prop} which relates meromorphic line bundles to log Picard algebroids: 
\begin{Proposition}	\label{dlog PA mero prop}
Let $\mathcal{M}$ be a meromorphic line bundle on $(X,D)$ and let $\AA_{\mathcal{M}}$ be its associated Atiyah algebroid. Then $\dlog_*[\MM] = [\AA_\MM]$, where $[\MM]\in H^1(\OO_X^\times(*D))$ and $[\AA_\MM]\in H^1(\Omega^{1,\mathrm{cl}}_X(\log D))$ are the corresponding isomorphism classes and $\dlog_*$ is induced by the homomorphism of sheaves of abelian groups
\[
\xymatrix{\OO_X^\times(*D)\ar[r]^-{\dlog}& \Omega_X^{1,\mathrm{cl}}(\log D)}.
\]
\end{Proposition}
%
%
%
%
%
In the remainder of this section, we solve the prequantization problem for log Picard algebroids on $(X,D)$ where $D$ is a smooth divisor in $X$. 

\begin{Lemma} \label{mero NS lemma}
There is an isomorphism of groups
\[
W_3H^{2}(U,\ZZ(1))\cong W_3H^2(U,\CC)\times_{H^2(U,\CC)}H^2(U,\ZZ(1)).
\]
\end{Lemma}

\begin{proof}
Recall from~\eqref{eq:W3 as hyperH2} that we have the isomorphism
\[
W_3H^2(U,\ZZ(1))\cong\HH^2(\sigma_{\leq 1}Rj_{*}\ZZ(1)_U),
\]
and analogously, we have the same isomorphism for complex cohomology
\[
W_3H^2(U,\CC)\cong\HH^2(\sigma_{\leq 1}Rj_{*}\CC_U).
\]
We have the following commutative diagram
\[
\xymatrix{
	0\ar[r] & \HH^2(\sigma_{\leq 1}Rj_{*}\ZZ(1)_U)\ar[r]\ar[d] & H^2(U,\ZZ(1))\ar[r]\ar[d] & \HH^2(\sigma_{>1}Rj_{*}\ZZ(1)_U)\ar[r]\ar[d] & 0 \\
	0\ar[r] & \HH^2(\sigma_{\leq 1}Rj_{*}\CC_U)\ar[r] & H^2(U,\CC)\ar[r] & \HH^2(\sigma_{>1}Rj_{*}\CC_U)\ar[r] & 0
}
\]
Now, the groups $\HH^2(\sigma_{>1}Rj_{*}\ZZ(1)_U)=H^0(R^2j_{*}\ZZ(1)_U)$ and $\HH^2(\sigma_{>1}Rj_{*}\CC_U)=H^0(R^2j_{*}\CC_U)$. We have $H^0(R^2j_{*}\ZZ(1)_U)=H^0(\ZZ_{D^{(2)}})$ from~\eqref{eq:RM=ZD}, and $H^0(R^2j_{*}\CC_U)=H^{0}(\CC_{D^{(2)}})$ from~\eqref{eq:R^k interpretation for C}. This implies that the map
\[
\HH^2(\sigma_{>1}Rj_{*}\ZZ(1)_U)\rightarrow\HH^2(\sigma_{>1}Rj_{*}\CC_U)
\]
is injective and we use the above diagram to conclude the result.
\end{proof}

\begin{Theorem}	\label{log PA prequantization theorem} 
Let $\AA$ be a locally trivial log Picard algebroid on $(X,D)$. There exists a prequantization for $\AA$ to a meromorphic line bundle $\LL$ if and only if the restriction of the class of $\AA$, $[\AA]\in H^{1}(\Omega_{X}^{1,\mathrm{cl}}(\log D))$, to the complement $U=X\setminus D$ has integral periods.
\end{Theorem}

\begin{proof}
From the above proposition, we know that there exists a prequantization for $\AA$ if and only if there exist a lift of the class $[\AA]$ under the following map
\[
H^{1}(\OO_{X}^{\times}(*D))\overset{\text{dlog}}{\rightarrow} H^{1}(\Omega_{X}^{1,\mathrm{cl}}(\log D))
\]
We show that this lift exists if and only if the restriction of $[\AA]$ to the complement $U$ has integral periods. Using Lemma~\ref{mero NS lemma}, we have the following isomorphism
\begin{equation*}
\mathrm{NS}(X,D)  \cong F^1H^2(U,\CC)\times_{H^{2}(U,\CC)}(W_3H^2(U,\CC)\times_{H^{2}(U,\CC)}H^2(U,\ZZ(1)) 
\end{equation*}
Consider the following commutative diagram
\[
\xymatrix{
 H^1(\Omega_X^{1,\mathrm{cl}}(\log D))\ar[r]^-\cong & F^1H^2(U,\CC)\cap W_3H^2(U,\CC) \\
 H^1(\OO_X^\times(*D))\ar[u]^{\mathrm{dlog}}\ar@{->>}[r]  & \mathrm{NS}(X,D)\ar[u]
}
\]
where the top horizontal isomorphism comes from the proof of Theorem~\ref{loc trivial log PA Hodge thm}. Suppose that $[\AA]$ has a lift under the dlog map, then the above commutative diagram shows that this class must have integral periods. Conversely, suppose that the restriction of $[\AA]$ to the complement $U$ has integral periods; by the above description of the meromorphic N\'eron-Severi group, this means that $[\AA]$ lies in the image of 
\[
\mathrm{NS}(X,D)\rightarrow F^1H^2(U,\CC)\cap W_3H^2(U,\CC).
\]
Hence, using the above diagram again, we obtain the existence of a lift of $[\AA]$ under the dlog map.

\end{proof}

\begin{Corollary}
The prequantizable log Picard algebroids on $(X,D)$ are classified by the meromorphic N\'eron-Severi group $\mathrm{NS}(X,D)$ modulo torsion.
\end{Corollary}

\begin{Proposition}
The meromorphic line bundles prequantizing a fixed log Picard algebroid on $(X,D)$ form a torsor for the group $\mathcal{K}$ of meromorphic line bundles with torsion first Chern class in $U = X\setminus D$, which is an extension 
\[
0\rightarrow\text{Jac}(X,D)\rightarrow\mathcal{K}\rightarrow H^{2}(U,\ZZ(1))_{\text{tors}}\rightarrow0
\]
of the torsion group by the meromorphic Jacobian.
%
\end{Proposition}

\begin{proof}
We know that log Picard algebroids $\AA$ on $X$ that can be prequantized are the ones whose class $[\AA]$ can be lifted under the dlog map:
\[
H^{1}(\OO_{X}^{\times}(*D))\overset{\mathrm{dlog}}{\rightarrow} H^{1}(\Omega_{X}^{1,\mathrm{cl}}(\log D)).
\]
The sheaf $\Omega_X^{1,\mathrm{cl}}(\log D)$ fits into the exact sequence,
\[
0\rightarrow\CC_{X}^{\times}\rightarrow\OO_{X}^{\times}(*D)\overset{\text{dlog}}{\rightarrow}\Omega_{X}^{1,\text{cl}}(\log D)\rightarrow\CC_{D}^{\times}\rightarrow0
\]
which gives the quasi-isomorphism
\[
\Omega_{X}^{1,\mathrm{cl}}(\log D)\sim\mathrm{Cone}(\sigma_{\leq 1}Rj_{*}\CC_{U}^\times\rightarrow\OO_{X}^{\times}(*D)).
\]
The long exact sequence associated to the above mapping cone is then:
\[
\dots\rightarrow H^{0}(\Omega_{X}^{1,\text{cl}}(\log D))\rightarrow H^{1}(U,\CC^{\times})\rightarrow H^{1}(\OO_{X}^{\times}(*D))\overset{\text{dlog}}{\rightarrow} H^{1}(\Omega_{X}^{1,\text{cl}}(\log D))\rightarrow\dots,
\]
where we place $H^1(U,\CC^\times)$ in second position because $\HH^1(\sigma_{\leq 1}Rj_{*}\CC_{U}^\times) = \HH^1(Rj_{*}\CC_{U}^\times) = H^{1}(U,\CC^{\times})$.
Therefore, two lifts of the class $[\AA]$ differ by a class in $H^{1}(\OO_{X}^{\times}(*D))$ lying in the image $\mathcal{K}$ of the map
\begin{equation}	\label{eq:image group II}
H^{1}(U,\CC^{\times})\rightarrow H^{1}(\OO_{X}^{\times}(*D)).
\end{equation}
From the short exact sequence
\[
0\rightarrow\ZZ(1)_{U}\rightarrow\CC_{U}\rightarrow\CC_{U}^{\times}\rightarrow0,
\]
we obtain the following short exact sequence:
\[
0\rightarrow\frac{H^{1}(U,\CC)}{H^{1}(U,\ZZ(1))}\rightarrow H^{1}(U,\CC^{\times})\rightarrow H^{2}(U,\ZZ)_{\text{tors}}\rightarrow0.
\]
Combining this with the long exact sequence given in~\eqref{eq:cone for OX(D) les II}:
\[
\dots\rightarrow H^{1}(U,\ZZ(1))\rightarrow H^{1}(\OO_{X})\rightarrow H^{1}(\OO_{X}^{\times}(*D))\rightarrow W_{3}H^{2}(U,\ZZ(1))\rightarrow\dots
\]
we find that the image of the map \eqref{eq:image group II} surjects onto $H^{2}(U,\ZZ(1))_{\text{tors}}$, which is contained in $W_{3}H^{2}(U,\ZZ(1))$, with kernel
\[
\frac{H^{1}(\OO_{X})}{\text{im}(H^{1}(U,\ZZ(1))\rightarrow H^{1}(\OO_{X}))}.
\]
This is precisely the meromorphic Jacobian $\text{Jac}(X,D)$ of $(X,D)$, giving the result.
\end{proof}

\begin{Example}
We now give an important non-example: an integral 2-form with the wrong Hodge weight (weight 4 in this case) and which therefore does not correspond to a meromorphic line bundle. Consider $(\PP^2,D)$ where $D$ is the union of three projective lines in triangle position as in Example~\ref{Deligne line bundle example}. The divisor complement is $U=\CC^*\times\CC^* = \mathrm{Spec}\,\CC[x^{\pm 1},y^{\pm 1}]$ and so has $H^{2}(U,\ZZ) \cong \ZZ$.  This group is generated by the first Chern class of the Deligne line bundle~\cite{Deligne91} on $U$, which may be defined as follows (see \cite{Brylinski00}). 
Express $U$ as the quotient of $\tilde{U} = \CC^{*}\times \CC = \mathrm{Spec}\,\CC[x^{\pm 1}, w = \log y]$ by the $\ZZ$-action generated by $(x, w)\sim (x, w+2\pi i)$, and consider the trivial line bundle $\tilde{U}\times \CC$ equipped with the $\ZZ$-equivariant structure 
\[
((x,w),t)\sim ((x,w+2\pi i), x^{-1}t).
\]
This descends to the required holomorphic line bundle $\LL_{(x,y)}$ over $U$.  The connection
\[
\nabla = d + \tfrac{1}{2\pi i} w x^{-1}dx
\] 
descends to a connection on $\LL_{(x,y)}$, with first Chern form 
\[
-\tfrac{1}{4\pi^{2}}x^{-1}dx\wedge y^{-1}dy,
\]
a generator of $H^{2}(U,\ZZ)$ as claimed.
%
%
In Example~\ref{Deligne line bundle example}, we used this closed log 2-form to twist the Lie bracket on $\TT_{\PP^2}(-\log D)\oplus\OO_{\PP^{2}}$ to obtain a non-trivial log Picard algebroid on $(\PP^2,D)$. The meromorphic N\'eron-Severi group $\mathrm{NS}(\PP^2,D)=0$ and so by the above corollary, this log Picard algebroid cannot be prequantized to a meromorphic line bundle. In fact, if we use the sequence in~\eqref{eq:cone for OX(D) les II}, we find that the meromorphic Picard group is trivial, $\mathrm{Pic}(\PP^2,D)=0$. Thus, we find that the Deligne line bundle $\LL_{(f,g)}$ is an example of a non-trivial holomorphic line bundle on the complement $U$ which does not extend to a meromorphic line bundle on $(\PP^2,D)$.
\end{Example}

\subsection{Functorial properties and blowup}\label{blowupsec}
Let $\varphi:Y\rightarrow X$ be a morphism of projective varieties and $D$ a simple normal crossings divisor on $X$ such that $\varphi^{*}D$ is also a simple normal crossings divisor on $Y$. We define the pullback of meromorphic line bundles:

\begin{Definition}
Let $\mathcal{L}$ be a meromorphic line bundle on $(X,D)$ given by transition functions $\{g_{ij}\}\in H^{1}(\OO_{X}^{\times}(*D))$, the \textit{\textbf{pullback meromorphic line bundle}} $\varphi^{*}\mathcal{L}$ of $\mathcal{L}$ is the locally-free $\OO_{Y}(*\varphi^{*}D)$-module $\varphi^{*}\mathcal{L}$ with transition functions $\{\varphi^{*}g_{ij}\}\in H^{1}(\OO_{Y}^{\times}(*\varphi^{*}D))$.
\end{Definition}

We now prove an analogue of Theorem \ref{log PA blowup thm} in the context of meromorphic line bundles. This will be necessary when we give explicit constructions of meromorphic line bundles in Section \ref{sec 4}. 

\begin{Lemma}
Let $X$ be a smooth projective algebraic surface with first Betti number $b_{1}(X)=0$ and let $\pi:\widetilde{X}\rightarrow X$ be the blowup of $X$ at a point $p\in X$. Then the analytic Brauer groups of $X$ and $\widetilde{X}$ coincide.
\end{Lemma}

\begin{proof}
Since $b_{1}(X)=0$, we have $H^{1}(\OO_{X})=0$. Moreover, by Poincar\'e duality, we have $H^{3}(X,\ZZ)=0$. Hence, we can write the long exact sequence associated to the exponential sheaf sequence on $X$ as follows
\begin{equation}	\label{eq:exp sequence for X}
0\rightarrow\mathrm{Pic}(X)\rightarrow H^{2}(X,\ZZ)\rightarrow H^{2}(\OO_{X})\rightarrow H^{2}(\OO_{X}^{\times})\rightarrow0
\end{equation}
Recall again that the first Betti number is a birational invariant for smooth projective surfaces which means that we analogously have the exact sequence for $\widetilde{X}$
\[
0\rightarrow\mathrm{Pic}(\widetilde{X})\rightarrow H^{2}(\widetilde{X},\ZZ)\rightarrow H^{2}(\OO_{\widetilde{X}})\rightarrow H^{2}(\OO_{\widetilde{X}}^{\times})\rightarrow0
\]
We have $\mathrm{Pic}(\widetilde{X})=\mathrm{Pic}(X)\oplus\ZZ$ and $H^{2}(\widetilde{X},\ZZ)=H^{2}(X,\ZZ)\oplus\ZZ$. Moreover, by the fact that the plurigenus is a birational invariant along with Serre duality, we have $H^{2}(\OO_{\widetilde{X}})=H^{2}(\OO_{X})$. Hence, we can re-write the above sequence as
\[
0\rightarrow\mathrm{Pic}(X)\oplus\ZZ\rightarrow H^{2}(X,\ZZ)\oplus\ZZ\rightarrow H^{2}(\OO_{X})\rightarrow H^{2}(\OO_{\widetilde{X}}^{\times})\rightarrow0
\]
Therefore, we can fit the above sequence and the sequence in \eqref{eq:exp sequence for X} in the commutative diagram
\[
\xymatrix{
0\ar[r] & \mathrm{Pic}(X)\ar[r]\ar[d] & H^{2}(X,\ZZ)\ar[r]\ar[d] & H^{2}(\OO_{X})\ar[r]\ar@{=}[d] & H^{2}(\OO_{X}^{\times})\ar[r]\ar[d] & 0 \\
0\ar[r] & \mathrm{Pic}(X)\oplus\ZZ\ar[r] &  H^{2}(X,\ZZ)\oplus\ZZ\ar[r] &  H^{2}(\OO_{X})\ar[r] &  H^{2}(\OO_{\widetilde{X}}^{\times})\ar[r] & 0}
\]
where all the vertical arrows are induced by the pullback $\pi^{*}$. This shows that the map
\[
H^{2}(\OO_{X}^{\times})\rightarrow H^{2}(\OO_{\widetilde{X}}^{\times})
\]
is an isomorphism, as required.
\end{proof}

\begin{Theorem}	\label{mero line bundle blowup thm} 
Let $X$ be a smooth projective algebraic surface $X$ with $b_{1}(X)=0$, $D$ a simple normal crossings divisor in $X$, and $\pi:\widetilde{X}\rightarrow X$ the blowup of $X$ at a point $p\in D$. The pullback of meromorphic line bundles defines an equivalence between the meromorphic Picard group $\mathrm{Pic}(X,D)$ on $(X,D)$ and the meromorphic Picard group on $\mathrm{Pic}(\widetilde{X},\pi^{*}D)$ of $(\widetilde{X},\pi^{*}D)$ where $\pi^{*}D:=D+E$ and $E$ is the exceptional divisor of the blowup.
\end{Theorem}

\begin{proof}
Let $\mathcal{L}$ be a meromorphic line bundle on $(X,D)$, the assignment $\mathcal{L}\mapsto\varphi^{*}\mathcal{L}$ induces the map in cohomology
\[
\mathrm{Pic}(X,D)\rightarrow\mathrm{Pic}(\widetilde{X},\pi^{*}D)
\]
We wish to show that this map is an isomorphism. First, we realize that the blowup of the point $p$ on the divisor $D$ does not change the complement $U$. Using Theorem \ref{mero picard extension thm}, we have the exact sequence
\[
0\rightarrow\mathrm{Pic}^{\mathrm{alg}}(U)\rightarrow\mathrm{Pic}(\widetilde{X},\pi^{*}D)\rightarrow H^{1}(D,\ZZ)\oplus H^{1}(E,\ZZ)\rightarrow H^{2}(\mathcal{O}_{\widetilde{X}}^{\times})
\]
As $E\cong\mathbb{P}^{1}$, we have $H^{1}(E,\ZZ)=0$. Furthermore, from the above lemma, we have that $H^{2}(\OO_{\widetilde{X}}^{\times})$ is isomorphic to $H^{2}(\OO_{X}^{\times})$ so we can re-write the above sequence as
\begin{equation}	\label{eq:mero sequence for blowup}
0\rightarrow\mathrm{Pic}^{\mathrm{alg}}(U)\rightarrow\mathrm{Pic}(\widetilde{X},\pi^{*}D)\rightarrow H^{1}(D,\mathbb{Z})\rightarrow H^{2}(\OO_{X}^{\times})
\end{equation}
Recall from \eqref{eq: mero line bundle sequence II}, we have
\[
0\rightarrow\mathrm{Pic}^{\mathrm{alg}}(U)\rightarrow\mathrm{Pic}(X,D)\rightarrow H^{1}(D,\mathbb{Z})\rightarrow H^{2}(\OO_{X}^{\times})
\]
We can fit the above sequence and the sequence in \eqref{eq:mero sequence for blowup} to obtain the commutative diagram:
\[
\xymatrix{
0\ar[r] & \mathrm{Pic}^{\mathrm{alg}}(U)\ar[r]\ar@{=}[d] & \mathrm{Pic}(X,D)\ar[r]\ar[d] & H^{1}(D,\mathbb{Z})\ar[r]\ar@{=}[d] & H^{2}(\OO_{X}^{\times})\ar@{=}[d] \\
0\ar[r] & \mathrm{Pic}^{\mathrm{alg}}(U)\ar[r] & \mathrm{Pic}(\widetilde{X},\pi^{*}D)\ar[r] & H^{1}(D,\mathbb{Z})\ar[r] & H^{2}(\OO_{X}^{\times}) }
\]
This proves that the map
\[
\mathrm{Pic}(X,D)\rightarrow\mathrm{Pic}(\widetilde{X},\pi^{*}D)
\]
is an isomorphism which concludes the proof.
\end{proof}

\section{Explicit constructions} \label{sec 4}

In this section, we provide a method for constructing meromorphic line bundles on certain ruled surfaces equipped with simple normal crossing divisors. We then employ the method, together with the blowup constructions from Section~\ref{blowupsec} to construct all meromorphic line bundles on $\CC P^2$ with singularities along a smooth cubic curve.

\subsection{Meromorphic line bundles on Hirzebruch surfaces}

Let $\pi:X\rightarrow\mathbb{P}^{1}$ be a projective line bundle, and let $D\subset X$ be a simple normal crossings divisor which is finite over the base.  We define the discriminant locus $\Delta\subset\mathbb{P}^{1}$ to be the union of the critical values of $\pi|_{D}:D\rightarrow\mathbb{P}^{1}$ and $\pi(D^{\mathrm{sing}})$, the image of the singular points of $D$.  

As described in Figure~\ref{covering}, choose an oriented simple closed curve $\gamma\subset\mathbb{P}^{1}\backslash\Delta$, and denote by $\{V_0,V_1\}$ an open covering of $\CC P^1$ such that $V_0$ contains the oriented interior of $\gamma$, $V_1$ contains the exterior, and $V_0\cap V_1$ is an annular neighbourhood $A_\gamma$ of $\gamma$ with no intersection with $\Delta$.  The inverse images $U_0=\pi^{-1}(V_0)$ and $U_1=\pi^{-1}(V_1)$ then form an open covering of $X$ with intersection $U_0\cap U_1 = \pi^{-1}(A_\gamma)$.

\begin{figure}[h]
\begin{tikzpicture}[scale=.5]
\begin{scope}[thick,decoration={
    markings,
    mark=at position 0.5 with {\arrow{>}}}
    ] 
\draw[postaction={decorate}] (48-49,-2) to [closed,curve through={(50-49,0) .. (52-49,1) .. (51-49,3) .. (47-49,3) .. (46-49.6,2) .. (45-49,0.8) }] (48-49,-2);
\end{scope}
\draw[line width=12,opacity=.1,] (48-49,-2) to[closed,curve through={(50-49,0) .. (52-49,1) .. (51-49,3) .. (47-49,3) .. (46-49.6,2) .. (45-49,0.8) }] (48-49,-2);
\node at (44-49,0){$\gamma$};
\node at (47-49,-1){$\times$};
\node at (48-49,2){$\times$};
\node at (53-49,3){$\times$};
\node at (51-49,-2){$V_1$};
\node at (54-49,-1.5){$\times$};
\node at (56-49,0.6){$\times$};
\node at (-1.5,0.5){$V_0$};
\end{tikzpicture}
\caption{A covering of $\CC P^1$ by discs $V_0,V_1$ intersecting in an annular neighbourhood of an oriented simple closed curve $\gamma$. The annulus avoids the discriminant locus.}\label{covering}
\end{figure}
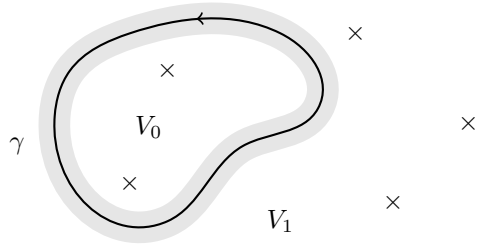

In this way, the choice of the simple closed curve $\gamma\subset\mathbb{P}^{1}\backslash\Delta$ defines an analytic open covering of $X$. We now construct a meromorphic line bundle by specifying a \v{C}ech cocycle for this covering.  That is, we require a function
\[
g_{01}\in \mathcal{O}_{X}^{\times}(*D) (U_0\cap U_1).
\]

As the annulus $A_{\gamma}$ does not intersect $\Delta$, the restriction $D_{01} = D\cap (U_0\cap U_1)$ is an unramified covering of $A_\gamma$. 
We now assume that there are two disjoint connected components $\gamma^+, \gamma^-$ of $D\cap \pi^{-1}(\gamma)$ which are isomorphic as covering spaces of $\gamma$; as a result they lie on linearly equivalent connected components of $D_{01}$ cut out by holomorphic sections $s_{+}, s_-$ respectively.  The resulting meromorphic function 
\[
g_{01} = \frac{s_{+}}{s_{-}}
\]
gives the required cocycle. We denote the corresponding meromorphic line bundle by $\mathcal{M}(\gamma^+,\gamma^-)$.

In this situation, the analytic Brauer group of $X$ vanishes, so by Theorem \ref{mero picard extension thm}, the meromorphic Picard group $\mathrm{Pic}(X,D)$ maps surjectively to the moduli of $\ZZ$-local systems on $D$ via the order map:
\begin{equation}	\label{eq:mero line bundle to Z-local system}
\mathrm{Pic}(X,D)\overset{\mathrm{ord}_{*}}{\rightarrow} H^{1}(D,\ZZ)
\end{equation}
The image of $\mathcal{M}(\gamma^+,\gamma^-)$ under the above order map is a $\ZZ$-local system on $D$ described by the 1-cocycle which is  $+1$ on $D^{+}$, $-1$ on $D^{-}$, and vanishes on any other components of $D\cap (U_0\cap U_1)$. Equivalently, for any 1-cycle $\beta$ in $D$, 
\[
\left<\mathrm{ord}_*(\mathcal{M}(\gamma^+,\gamma^-)), [\beta] \right> =  \gamma^{+}\cdot\beta -\gamma^{-}\cdot\beta, 
\]
where $\gamma^{\pm}=\pi^{-1}(\gamma)\cap D^{\pm}$ are equipped with the orientation inherited from $\gamma$.
We summarize this construction in the following theorem. 

\begin{Theorem}\label{mero line bundle local system theorem}
Given an oriented simple closed curve $\gamma\subset\mathbb{P}^{1}\backslash\Delta$ having  disjoint connected components $\gamma^+, \gamma^-$ of $D\cap \pi^{-1}(\gamma)$ of the same degree over $\gamma$, the above construction defines a meromorphic line bundle $\mathcal{M}(\gamma^+,\gamma^-)$ whose corresponding local system is Poincar\'e dual to $[\gamma^+]-[\gamma^-]$, i.e.,
\[
\mathrm{ord}_*\mathcal{M}(\gamma^+,\gamma^-) = \mathrm{PD}([\gamma^+] - [\gamma^-]).
\]
\end{Theorem}

\begin{Example}\label{hirzex}
Consider $(\mathbb{P}^{2},D)$ where $D$ is the smooth cubic curve in $\mathbb{P}^{2}$. If we blow-up a point $p\in D$, we obtain the first Hirzebruch surface $\mathbb{F}_{1}$ with the divisor $D+E$ where $E\cong\mathbb{P}^{1}$ is the exceptional divisor of the blow-up. We will now work with $(\mathbb{F}_{1},D+E)$. Since the analytic Brauer group $H^{2}(\OO_{\mathbb{F}_{1}}^{\times})$ is trivial, by Theorem \ref{mero picard extension thm}, the meromorphic Picard group $\mathrm{Pic}(\mathbb{F}_{1},D+E)$ maps surjectively to the moduli of $\ZZ$-local systems on $D+E$ via the order map
\[
\mathrm{Pic}(\mathbb{F}_{1},D+E)\overset{\mathrm{ord}_{*}}{\rightarrow} H^{1}(D,\ZZ)\oplus H^{1}(E,\ZZ).
\]
There are no non-trivial $\ZZ$-local systems on $E$ since $E\cong\mathbb{P}^{1}$. Hence, we have the surjection 
\[
\mathrm{Pic}(\mathbb{F}_{1},D+E)\overset{\mathrm{ord}_{*}}{\rightarrow} H^{1}(D,\ZZ).
\]

The first Hirzebruch surface $\mathbb{F}_{1}$ can be realized as a $\mathbb{P}^{1}$-bundle over $\mathbb{P}^{1}$ where the divisor $D+E$ on $\mathbb{F}_{1}$ is a 3-to-1 cover of $\mathbb{P}^{1}$ ramified at four points on $\mathbb{P}^{1}$ which we label as $p_{1},p_{2},p_{3},p_{4}$. The divisor $D$ is the standard 2-to-1 cover of $\PP^{1}$ ramified at the four points. Consider the oriented simple closed curve $\gamma_{12}$ containing the points $p_{1},p_{2}$ in its interior and two disjoint connected components $\gamma_{12}^{+},\gamma_{12}^{-}$ of $(D+E)\cap\pi^{-1}(\gamma_{12})$ where $\gamma_{12}^{+}$ corresponds to one branch of $D$ over $\gamma_{12}$ and $\gamma_{12}^{-}$ corresponds to the branch of $E$ over $\gamma_{12}$. By the above theorem, this defines a meromorphic line bundle $\MM(\gamma_{12}^{+},\gamma_{12}^{-})$ and the order of this meromorphic line bundle, $\mathrm{ord}_{*}\MM(\gamma_{12}^{+},\gamma_{12}^{-})$, is a local system that is Poincar\'e dual to the cycle $[\gamma_{12}^{+}]$ in $H_{1}(D,\ZZ)$.

If we now repeat the same argument as above for a oriented simple closed curve $\gamma_{23}$ containing the points $p_{2},p_{3}$ in its interior, we obtain the meromorphic line bundle $\MM(\gamma_{23}^{+},\gamma_{23}^{-})$ where the order of this bundle, $\mathrm{ord}_{*}\MM(\gamma_{23}^{+},\gamma_{23}^{-})$, is a local system that is Poincar\'e dual to the cycle $[\gamma_{23}^{+}]$ in $H_{1}(D,\ZZ)$.

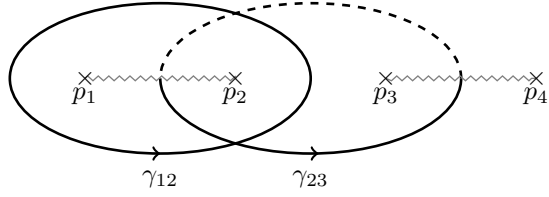
\begin{figure}[h]
	\centering
	\begin{tikzpicture}[scale=1]
	\tikzset{
		partial ellipse/.style args={#1:#2:#3}{
			insert path={+ (#1:#3) arc (#1:#2:#3)}
		}
	}
	\begin{scope}[thick,decoration={
    markings,
    mark=at position .75 with {\arrow{>}}}
    ] 
	\draw[line width=1pt,postaction={decorate}] (-2,0) ellipse (2 and 1);
	\end{scope}

	\draw[dashed, line width=1pt] (0,0) [partial ellipse=0:180:2 and 1];
		\begin{scope}[thick,decoration={
    markings,
    mark=at position .5 with {\arrow{<}}}
    ] 
	\draw[line width=1pt,postaction={decorate}] (0,0) [partial ellipse=0:-180:2 and 1];
	\end{scope}
	\draw[line width=.5pt,color=gray,decorate, decoration={zigzag, segment length=4,amplitude=.9}] (-3,0)  -- (-1,0)	(1,0) -- (3,0);
	\fill (-3,0) circle (0pt) node[below]{$p_1$}
	(-1,0) circle (0pt) node[below]{$p_2$}
	(1,0) circle (0pt)  node[below]{$p_3$}
	(3,0) circle (0pt)  node[below]{$p_4$};
	\node at (-3,0){$\times$};
	\node at (-1,0){$\times$};
	\node at (1,0){$\times$};
	\node at (3,0){$\times$};
	
	\node[below] at (-2,-1.1) {$\gamma_{12}$};
	\node[below] at (0, -1.1) {$\gamma_{23}$};
	
	\end{tikzpicture}
	\caption{The oriented simple closed curves $\gamma_{12}, \gamma_{23}$ on $\PP^{1}$ which give rise to cycles $\gamma_{12}^{+}, \gamma_{23}^{+}$ on the genus 1 curve $D$.}
	\label{P1 four points}
\end{figure}

Each of the cycles $\gamma_{12}^{+}$ and $\gamma_{23}^{+}$ have trivial self-intersection and $\gamma_{12}^{+}\cdot\gamma_{23}^{+} = 1$. Therefore, the two cycles together form a homology basis for $H_{1}(D,\ZZ)$. As a result, the local systems $\mathrm{ord}_{*}\MM(\gamma_{12}^{+},\gamma_{12}^{-})$ and $\mathrm{ord}_{*}\MM(\gamma_{23}^{+},\gamma_{23}^{-})$ coming from the meromorphic line bundles $\mathrm{ord}_{*}\MM(\gamma_{12}^{+},\gamma_{12}^{-})$ and $\mathrm{ord}_{*}\MM(\gamma_{23}^{+},\gamma_{23}^{-})$ respectively generate the cohomology group $H^{1}(D,\ZZ)$.
\end{Example}

\subsection{Meromorphic line bundles for the smooth cubic in $\mathbb{P}^{2}$} 

We conclude the paper by constructing all meromorphic line bundles on $(\mathbb{P}^{2},D)$ where $D$ is the smooth cubic curve in $\mathbb{P}^{2}$. Let $\pi:\mathbb{F}_1\to\PP^2$ be the blowup of a point $p\in D$ and $E$ the exceptional divisor. By Theorem \ref{mero line bundle blowup thm}, we know that the pullback defines an equivalence between $\mathrm{Pic}(\PP^{2},D)$ and $\mathrm{Pic}(\mathbb{F}_{1},D+E)$. Using the results of Example~\ref{hirzex}, we uniquely determine meromorphic line bundles $\MM_{1}$, $\MM_{2}$ on $(\PP^{2},D)$ such that $\pi^{*}\MM_{1} = \MM(\gamma_{12}^{+},\gamma_{12}^{-})$ and $\pi^{*}\MM_{2} =\MM(\gamma_{23}^{+},\gamma_{23}^{-})$.  
In Example~\ref{curve in P2 mero}, we showed that the meromorphic Picard group of $(\PP^{2},D)$ is given by
\[
\mathrm{Pic}(\PP^{2},D) = \ZZ_{3}\oplus\ZZ^{2}. 
\] 
We now explain how the meromorphic line bundles $\MM_1,\MM_2$, together with the tautological line bundle, generate the meromorphic Picard group. 
\begin{Theorem}
Any meromorphic line bundle on $(\mathbb{P}^{2},D)$ is isomorphic to exactly one of
\[
\OO_{\PP^{2}}(1)^{k}\otimes \mathcal{M}_{1}^{l_1}\otimes \mathcal{M}_{2}^{l_2},
\]
where $k\in\ZZ_3$ and $l_{1},l_{2}\in\ZZ$.
\end{Theorem}

\begin{proof}
As the pullbacks of the meromorphic line bundles $\MM_{1}$ and $\MM_{2}$ on $(\PP^{2},D)$ are the meromorphic line bundles $\MM(\gamma_{12}^{+},\gamma_{12}^{-})$  and $\MM(\gamma_{23}^{+},\gamma_{23}^{-})$ on $(\mathbb{F}_{1},D+E)$ respectively, the local systems corresponding to $\MM_{1}$ and $\MM_{2}$ under the order map generates all of $H^{1}(D,\ZZ)$. If the order is trivial, then it must be in the image of the map
\[
\mathrm{Pic}(\mathbb{P}^{2})\rightarrow\mathrm{Pic}(\mathbb{P}^{2},D).
\]
However, the image of the above map is the cokernel of the degree map
\begin{equation}	\label{eq:degree map cubic P2}
H^{0}(D,\ZZ)\overset{\cdot3}{\rightarrow}\text{Pic}(\mathbb{P}^{2}),
\end{equation}
which is the cyclic group $\ZZ_{3}$ generated by the line bundle $\OO_{\mathbb{P}^{2}}(1)$, as required.
\end{proof}

Observe that since we are working with the smooth cubic in $\mathbb{P}^{2}$ and $\mathbb{P}^{2}$ has vanishing Hodge numbers $h^{0,1}, h^{0,2}$,  Corollary \ref{mero line bundle=H2(U,Z)} implies that the meromorphic Picard group $\mathrm{Pic}(\mathbb{P}^{2},D)$ is isomorphic to the analytic Picard group $\mathrm{Pic}(U^{\mathrm{an}})$ of the complement $U$.  On the other hand, the algebraic Picard group of $U$ can be characterized as the cokernel of the degree map in \eqref{eq:degree map cubic P2} and so $\text{Pic}^{\text{alg}}(U)\cong\ZZ_{3}$, with generator $\OO_{\mathbb{P}^{2}}(1)|_{U}$. Thus, if we restrict the meromorphic line bundles $\mathcal{M}_{1}$ and $\mathcal{M}_{2}$ on $(\mathbb{P}^{2},D)$ to $U$, these give holomorphic line bundles on $U$ that are not algebraic.

\bibliographystyle{plain}
\bibliography{bibliography}

\end{document}